\documentclass[11pt]{article}

\RequirePackage{amsthm,amsmath,amsfonts,amssymb}
\RequirePackage[authoryear]{natbib}
\RequirePackage[colorlinks,citecolor=blue,urlcolor=blue]{hyperref}
\RequirePackage{graphicx}

\usepackage{amscd, paralist}
\usepackage{enumerate}
\usepackage{enumitem}
\usepackage[export]{adjustbox}
\usepackage[dvipsnames]{xcolor}
\usepackage{bm}
\usepackage[title]{appendix}
\usepackage{tikz}
\usepackage{tikz-cd}
\usepackage{wrapfig}
\usetikzlibrary{matrix,arrows,positioning}
\usepackage{lscape}
\usepackage{algorithm2e}
\usepackage{float}
\usepackage[normalem]{ulem}
\usepackage{comment}

\usetikzlibrary{shapes,decorations,arrows,calc,arrows.meta,fit,positioning}
\tikzset{
    -Latex,auto,node distance =1 cm and 1 cm,semithick,
    state/.style ={ellipse, draw, minimum width = 0.7 cm},
    point/.style = {circle, draw, inner sep=0.04cm,fill,node contents={}},
    bidirected/.style={Latex-Latex,dashed},
    el/.style = {inner sep=2pt, align=left, sloped}
}

\theoremstyle{plain}

\newtheorem{theorem}{Theorem}[section]

\newtheorem{lemma}[theorem]{Lemma}

\theoremstyle{definition}

\newtheorem{assumption}[theorem]{Assumption}

\definecolor{Red}{RGB}{225,0,0}

\definecolor{Blue}{RGB}{0,0,255}

\definecolor{Cyan}{RGB}{0,180,255}

\definecolor{Green}{RGB}{0,160,0}

\definecolor{Alert}{RGB}{255,122,0}

\definecolor{NavyBlue}{RGB}{0,100,175}

\definecolor{NavyRed}{RGB}{125,0,0}

\newcommand{\signot}{\sigma}

\newcommand{\eps}{\varepsilon}

\newcommand{\wt}{\widetilde}
\newcommand{\wh}{\widehat}

\newcommand{\argmin}{\mathop{\rm arg\min}}
\newcommand{\argmax}{\mathop{\rm arg\max}}

\newcommand{\cor}{\operatorname{cor}}

\newcommand{\Tr}{\operatorname{Tr}}

\newcommand{\rank}{\operatorname{rk}}

\newcommand{\diag}{\operatorname{diag}}

\newcommand{\ls}{\operatorname{LS}}

\newcommand{\pls}{\operatorname{PLS}}
\newcommand{\spa}{\operatorname{span}}

\newcommand{\E}{\mathbb{E}}

\newcommand{\R}{\mathbb{R}}

\newcommand{\mB}{\mathcal{B}}

\newcommand{\mK}{\mathcal{K}}

\newcommand{\mM}{\mathcal{M}}
\newcommand{\mN}{\mathcal{N}}

\newcommand{\mR}{\mathcal{R}}

\newcommand{\bzero}{\mathbf{0}}

\newcommand{\bA}{\mathbf{A}}
\newcommand{\bb}{\mathbf{b}}
\newcommand{\bB}{\mathbf{B}}

\newcommand{\be}{\mathbf{e}}

\newcommand{\bI}{\mathbf{I}}
\newcommand{\bk}{\mathbf{k}}
\newcommand{\bK}{\mathbf{K}}

\newcommand{\bP}{\mathbf{P}}
\newcommand{\bq}{\mathbf{q}}
\newcommand{\bQ}{\mathbf{Q}}
\newcommand{\br}{\mathbf{r}}

\newcommand{\bU}{\mathbf{U}}

\newcommand{\bw}{\mathbf{w}}

\newcommand{\bx}{\mathbf{x}}
\newcommand{\bX}{\mathbf{X}}
\newcommand{\by}{\mathbf{y}}

\newcommand{\balpha}{\bm{\alpha}}
\newcommand{\bbeta}{\bm{\beta}}

\newcommand{\beps}{\bm{\eps}}

\newcommand{\bzeta}{\bm{\zeta}}

\newcommand{\bsigma}{\boldsymbol{\sigma}}
\newcommand{\bSigma}{\boldsymbol{\Sigma}}

\newcommand{\footremember}[2]{%
   \footnote{#2}
    \newcounter{#1}
    \setcounter{#1}{\value{footnote}}%
}
\newcommand{\footrecall}[1]{%
    \footnotemark[\value{#1}]%
}

\begin{document}

\RestyleAlgo{ruled}
\SetKwComment{Comment}{/* }{ */}
\SetKw{KwIn}{\textbf{Input:}}
\SetKw{KwOut}{\textbf{Output:}}
\SetKw{And}{\textbf{and}}


\title{An extended latent factor framework \\ for ill-posed linear regression}

\author{%
    Gianluca Finocchio\footremember{uv}{Department of Statistics and Operations Research, Universit\"at Wien, Oskar-Morgenstern-Platz 1, 1090 Wien, Austria}%
    \ and Tatyana Krivobokova\footrecall{uv}%
}

\date{\today}
\maketitle

\begin{abstract}
In many applications, particularly in the natural sciences, the available high-dimen\-sional set of features may contain variables that are not correlated with the response under consideration. Such irrelevant features can, in certain cases, hinder both the accurate estimation and meaningful interpretation of the effects of the relevant features on the response. At the same time, the relevant features may also be well-approximated within a low-dimensional linear subspace, rendering the problem ill-posed. These observations motivate an extension of the classical latent factor model for linear regression. In this extended framework, it is assumed that, up to an unknown orthogonal transformation, the feature set comprises two subsets: one relevant and one irrelevant to the response. A joint low-dimensionality is imposed solely on the relevant features and the response variable. This setting enables the analysis of arbitrary linear dimensionality reduction techniques under a random design setting. In particular, it is demonstrated why principal component regression (PCR) is generally unsuitable for most applications. The framework also allows for a comprehensive analysis of the partial least squares (PLS) algorithm under random design. High-probability convergence rates are established for the sample PLS estimator with respect to an oracle latent coefficient vector, along with the corresponding linear prediction risk. Additionally, it is shown that early stopping can be guided by the empirical condition numbers of the projected design matrix. The theoretical results are validated through numerical studies on both real and simulated datasets.

\end{abstract}


%

\section{Introduction}\label{sec:intro}
Many applications in natural sciences involve high-dimensional datasets consisting of a response vector and a design matrix of highly-correlated features. The goal of the practitioners is the identification of the linear combinations of features that can explain the response. Due to the ill-posedness of the problem, induced by the high correlation among the features, it is typically assumed that some low-dimensional latent variables determine both the features and the response. 
Thereby, it is typically overlooked that the available set of features might include also a subset, generally unknown and potentially high in variance, which is unrelated to the response variable of interest. 
For example, in genome-wide association studies (GWAS) reviewed by \cite{Uffelmann2021gen} practitioners observe the whole genome in order to study responses related to various diseases. Since the whole genome cannot be responsible for any single disease, it is reasonable to assume that many available features, also those having high variance, are uncorrelated with the response under consideration.

Another particularly prominent example are datasets obtained from molecular-dynamics (MD) simulations of biological systems, such as proteins, pioneered by \cite{Warshel1976} and \cite{McCammon1977} whose groundbreaking works led to the shared Nobel Prize in Chemistry 2013 awarded by \cite{Nobel2013}. These numerically simulated datasets $(\bX,\by)\in\R^{n\times p}\times\R^n$ consist of $n\geq1$ synthetic configurations $\bx_{i}\in\R^{3N}$ of $N\geq1$ atoms in the Euclidean space, thus $p=3N$ spatial coordinates, and functional quantities of interest $y_{i}\in\R$, for all $1\leq i\leq n$, which can be the distance between two sub-regions, the volume of a sub-region or any other geometric or physical observable. Again, depending on the goal of the study, different subsets of the protein atoms may be related to the response. For example, to explain a distance between two sub-regions of a protein, it is reasonable to assume that only atoms in vicinity to those regions can provide useful information. Thereby, the highest variation may still be attributed to the unrelated atoms.

In the regression setting with datasets from MD simulations, practitioners strive to identify linear reductions of the design matrix $\bX\in\R^{n\times p}$ that also preserve the information on the functional quantity $\by\in\R^n$. To this end, they have applied Principal Components Analysis (PCA) by \cite{Pearson1901} to estimate the leading collective motions of proteins, see the non-exhaustive list of works by \cite{Garca1992}, \cite{Amadei1993}, \cite{Berendsen2000}, \cite{Alakent2004} and \cite{Hub2009}. The use of PCA has become so prevalent that specific reviews on this topic have been published by \cite{David2013}, \cite{Kitao2022}, \cite{Palma2022} and \cite{Moradi2024}. An alternative procedure is Partial Least Squares (PLS) by \cite{Wold66non}, but only a few works by \cite{krivobokova2012} and the same research group rely on the PLS algorithm. 

As a case study, we revisit the findings of \cite{krivobokova2012} who considered data generated by the MD simulations for the yeast aquaporin (Aqy1), the gated water channel of the yeast \textit{Pichia pastoris}. The data are given as Euclidean coordinates of $N=783$ atoms observed at $n=20.000$ equidistant time points, together with the diameter of the channel $y_{i}$ measured by the distance between two centers of mass of certain residues of the protein $\bx_{i}$. The authors showed that the linear model $\by=\bX\bbeta+\beps$ is well-specified and the aim of the analysis was to identify the collective motions of the atoms that is maximally correlated to the channel opening in the sense of \cite{Hub2009}. The authors also compared the performance of PLS and Principal Component Regression (PCR) in such setting and found that only PLS but not PCR is able to detect important directions of motion. In a later work, \cite{singer2016partial} studied the PLS algorithm positing for the Aqy1 dataset the population latent factor model
\begin{align*}
    \bx = \bP\bq + \be,\quad y=\bq^t\balpha+\eps,
\end{align*}
for some random vector $\bq\in\R^m$, deterministic matrix $\bP\in\R^{m\times p}$ and vector $\balpha\in\R^m$, suitable random residuals $\be\in\R^p$ and $\eps\in\R$. Such latent factor models have been studied in detail by \cite{stock2002} and \cite{bai2002} and, under the regularity conditions discussed by \cite{fan2023}, it has been shown that the PCR method consistently estimates the latent structure. Under the same conditions, \cite{bing2021pre} established the finite-sample prediction risk of an adaptive PCR algorithm. Since latent factor models implicitly assume that only projections of the features along their main directions of variation matter for the response, there is no theoretical reason to believe that PLS should outperform PCR in the identification of the latent factors. This is in contrast with the heuristic findings by \cite{krivobokova2012} where PLS strongly outperforms PCR.

Classical latent factor models implicitly assume that the residual $\be\in\R^p$ only accounts for small directions of variation, thus do not allow for projections of the features along large directions of variation to be uncorrelated with the response. We build upon the previous work by \cite{finocchio2025mod} and provide a novel framework that extends the scope of latent factor models by including projections $\bx_{y}$ and $\bx_{y^\bot}$ of the features $\bx$ that are relevant and irrelevant for the response $y$; a latent factor model on the relevant pair $(\bx_{y},y)$ so that
\begin{align*}
    \bx=\bx_{y}+\bx_{y^\bot},\quad \bx_{y}=\bP\bq+\be,\quad y=\bq^t\balpha+\eps,
\end{align*}
where $\bx_{y^\bot}$ is allowed to have arbitrarily large variation. Differently from classical latent factor models, the model in the above display explains the difference in performance observed by \cite{krivobokova2012} when comparing PLS and PCR on their Aqy1 dataset. In fact, at the population level, the main directions of variation might correspond to projections along irrelevant directions for the response, making PCR fail in general.

We exploit this new framework to make the following contributions. The first one is to formalize what we call \textit{parsimonious} linear reductions of the relevant features $\bx_{y}$, which we compute from directions of steepest-descent of the least-squares functional in population. These reductions factorize the relevant features in terms of their projections onto low-dimensional linear subspaces that preserve most of the information on the response. The second contribution is to provide a transparent characterization of the PLS algorithm and show that it is inherently built to consistently estimate the parsimonious linear reductions of the relevant features. We develop the theory of the PLS method in this general setting and infer both finite-sample convergence rate and prediction risk, building on a previous work by \cite{finocchio2025mod}. We thus extend the most notable and recent contributions on the statistical properties of the PLS algorithm under random design due to \cite{singer2016partial}, who established the finite-sample convergence rates under the classical latent factor model, and to \cite{Cook2019par}, who established the asymptotic prediction risk under a classical linear model.

\section{Extended Latent Factor Framework}\label{sec:elf}
In this section we develop a novel notion of parsimonious linear combinations of features that are important for a response variable. Since the response might be determined by projections of the features on a possibly low-dimensional linear subspace, we formally distinguish between projections of the features that are relevant and irrelevant for the response. We borrow the framework for ill-posed least-squares regression from \cite{finocchio2025mod}. We consider a dataset $(\bX,\by)\in\R^{n\times p}\times\R^n$ consisting of $n\geq1$ i.i.d. realizations $(\bx_{i},y_{i})$ of the same population pair $(\bx,y)\in\R^p\times\R$ under the following assumption. In what follows, we denote $\R_{\succeq0}^{p\times p}$ the space of matrices that are symmetric and positive semidefinite.

\begin{assumption}[Model-Free, 2nd moments]\label{ass:x.y.model.free.2nd}
    The features $\bx\in\R^p$ are a random vector and the response $y\in\R$ is a random variable, they are both centered and have finite second moments $\bSigma_{\bx}=\E(\bx\bx^t)\in\R_{\succeq0}^{p\times p}$, $\bsigma_{\bx,y}=\E(\bx y)\in\R^p\setminus\{\bzero_{p}\}$ and $\sigma_{y}^2=\E(y^2)>0$. The features are possibly degenerate with $1\leq r_{\bx} = \rank(\bSigma_{\bx}) \leq p$.
\end{assumption}

Regardless of the true dependence between the features and the response one can show, see Lemma~\ref{lem:x.y.ls}, that the population least-squares problem
\begin{align}\label{eq:x.y.ls.pop}
    \ls(\bx,y) := \argmin_{\bbeta\in\R^p} \E(\bx^t\bbeta - y)^2 = \argmin_{\bbeta\in\R^p} \left\{\bbeta^t\bSigma_{\bx}\bbeta - 2\bbeta^t\bsigma_{\bx,y} + \sigma_{y}^2 \right\} =:\ls(\bSigma_{\bx},\bsigma_{\bx,y})
\end{align}
admits minimum-$L^2$-norm solution $\bbeta_{\ls}:=\bSigma_{\bx}^\dagger\bsigma_{\bx,y}\in\R^p$ and only depends on the second moments of the population pair $(\bx,y)$. The features $\bx$ belong almost surely to the range $\mR(\bSigma_{\bx})$, see Lemma~\ref{lem:x.y.cov}, and it has been shown, see Lemma~\ref{lem:ls.rel.pop}, that one can uniquely define the \textit{relevant} subspace $\mB_{y}$ as the smallest linear subspace $\wt{\mB}$ of $\mR(\bSigma_{\bx})$ for which the projected features $\bx_{\wt{\mB}^\bot}$ along the complement $\wt{\mB}^\bot$ are uncorrelated with both the response $y$ and the the projection $\bx_{\wt{\mB}}$ of the features on $\wt{\mB}$. Formally, this corresponds to 
\begin{align}\label{eq:rel.sub}
    \mB_{y} := \argmin \Big\{\dim(\wt{\mB}) : \mR(\bSigma_{\bx})=\wt{\mB}\oplus\wt{\mB}^\bot,\ \E(\bx_{\wt{\mB}^\bot} y) = \bzero_{p},\ \E(\bx_{\wt{\mB}^\bot} \bx_{\wt{\mB}}^t) = \bzero_{p\times p}\Big\}.
\end{align}
We denote $\bU_{y}$ and $\bU_{y^\bot}$ the orthogonal projections onto $\mB_{y}$ and $\mB_{y}^\bot$, respectively. We call relevant features the projection $\bx_{y} := \bU_{y}\bx$ and irrelevant features the projection $\bx_{y^\bot} := \bU_{y^\bot}\bx$. This induces the orthogonal decompositions
\begin{align}\label{eq:x.rel.x.irr}
    \R^p = \mR(\bSigma_{\bx}) \oplus \mR(\bSigma_{\bx})^\bot,\quad \mR(\bSigma_{\bx}) = \mB_{y} \oplus \mB_{y}^\bot, \quad \bx = \bx_{y} + \bx_{y^\bot} \in \mR(\bSigma_{\bx}).
\end{align}  
By construction, see Lemma~\ref{lem:ls.rel.pop}, the relevant subspace preserves the population least-squares problem in Equation~\eqref{eq:x.y.ls.pop} in the sense that $\ls(\bx,y)=\ls(\bx_{y},y)$ and the population least-squares solution becomes $\bbeta_{\ls}=\bSigma_{\bx_{y}}^\dagger\bsigma_{\bx_{y},y}$. The latter only depends on the moments of the relevant population pair $(\bx_{y},y)$, thus the relevant subspace $\mB_{y}=\mR(\bSigma_{\bx_{y}})$ has dimension $r_{y}:=\rank(\bSigma_{\bx_{y}})$. When some low-dimensionality is at play, one expects the relevant subspace to be ill-posed in the sense that the condition number $\kappa_{2}(\bSigma_{\bx_{y}})$ is arbitrarily large. With $\lambda_{1}(\bSigma_{\bx_{y}})\geq\cdots\geq\lambda_{r_{y}}(\bSigma_{\bx_{y}})>\lambda_{r_{y}+1}(\bSigma_{\bx_{y}})=\cdots=\lambda_{p}(\bSigma_{\bx_{y}})=0$ the sorted eigenvalues of $\bSigma_{\bx_{y}}$, the condition number $\kappa_{2}(\bSigma_{\bx_{y}})=\lambda_{1}(\bSigma_{\bx_{y}})/\lambda_{r_{y}}(\bSigma_{\bx_{y}})$ is always the ratio between the largest and smallest non-zero eigenvalues.

Our goal is to define parsimonious linear combinations of the relevant features $\bx_{y}$ that factorize the population least-squares solution $\bbeta_{\ls}\in\mB_{y}$ in terms of its projections onto $s$-dimensional linear subspaces for all $1\leq s\leq r_{y}$ that capture as much as possible of the dependence between the features and the response, in a sense to be specified below.

\subsection{Extended Latent Factor Linear Models}\label{sec:lm}
There are many applications where a linear model can be posited as the underlying generating process in the sense that 
\begin{align}\label{eq:x.y.lm}
    y=\bx^t\bbeta+\eps
\end{align}
with $\bbeta\in\R^p$ a vector of effects and $\eps\in\R$ a random residual. An example is the molecular-dynamics simulation of Aqy1 studied, among other protein systems, by \cite{krivobokova2012} where the features are configurations of atoms in the Euclidean space and the response is the distance between two atoms in the same small region of space. Another example is the genome-wide association study on BMI by \cite{Locke2015gen} where the features are gene expressions of individuals and the response is the corresponding body mass index. The common strategy of these papers is to estimate the vector of effects $\bbeta$ in order to identify and interpret linear combinations of features that are important for the response. In such problems it is common for the features to be highly correlated, therefore it is more appropriate to posit a latent factor linear model
\begin{align}\label{eq:x.y.lfm}
    y=\bq^t\balpha+\eps,\quad \bx=\bq+\signot\be,
\end{align}
where $\bq\in\R^p$ is a centered random vector with rank $1\leq r_{\bq}=\rank(\bSigma_{\bq})\leq r_{\bx}$ and range $\mR(\bSigma_{\bq})\subseteq\mR(\bSigma_{\bx})$, $\balpha\in\R^p$ a vector of latent coefficients, $\eps\in\R$ an independent random variable that is centered, $\signot\geq0$ a noise parameter, $\be\in\R^p$ an independent random vector that is centered with full range $\mR(\bSigma_{\be})=\mR(\bSigma_{\bx})$ and normalized with $\lambda_{1}(\bSigma_{\be})=1$. The latter display holds without loss of generality since one recovers the classical linear model in Equation~\eqref{eq:x.y.lm} when $\bq=\bx$, $\balpha=\bbeta$ and $\signot=0$. To fix the ideas, in the Aqy1 dataset by \cite{krivobokova2012} the latent features are atoms in the vicinity of the region where the response is computed, whereas in the BMI dataset by \cite{Locke2015gen} the latent features are genotypes that correlate with body weight.

Under latent factor models one finds moments $\bsigma_{\bx,y}=\bsigma_{\bq,y}$ and $\bSigma_{\bx}=\bSigma_{\bq}+\signot^2\bSigma_{\be}$ and it is standard to assume that the noise level $\signot\geq0$ is sufficiently separated from the variance of the latent features $\bq$ in the sense that $\signot^2<\lambda_{r_{\bq}}(\bSigma_{\bq})$. All the information on the dependence between the features and the response is fully contained in the oracle linear subspace $\mR(\bSigma_{\bq})\subseteq\mR(\bSigma_{\bx})$ which is the $r_{\bq}$-dimensional range of the latent features. A small noise level makes the model ill-posed since the features are almost degenerate. When the noise level is sufficiently small, the $r_{\bq}$-dimensional principal eigenspace of $\bSigma_{\bx}$ is close to the range $\mR(\bSigma_{\bq})$ so that the response $y$ only depends on projections of the features $\bx$ along main directions of variation. 

When the features are high-dimensional, they typically contain a lot of information that is not useful for the specific response that is being studied. It is also unknown which combinations of features contain useful information on the response and it is too restrictive to assume that they are aligned with the directions of largest variation of the features. For the Aqy1 dataset studied by \cite{krivobokova2012} it is conceivable that only atoms that are in the same region of the response are relevant, whereas atoms that are further away are irrelevant despite having non-negligible variation. For the BMI dataset by \cite{Locke2015gen} it is conceivable that only gene expressions that are correlated with body weight are relevant for the response, whereas the others provide no information irrespective of their variation. Since the classical latent factor model does not allow for irrelevant features to have a large variation, we propose the extended latent factor linear model where only the relevant pair $(\bx_{y},y)$ satisfies Equation~\eqref{eq:x.y.lfm} and the features in Equation~\eqref{eq:x.rel.x.irr} become
\begin{align}\label{eq:x.y.elfm}
    y=\bq^t\balpha+\eps,\quad \bx=\bx_{y}+\bx_{y^\bot}=\bq+\signot\be+\bx_{y^\bot}, 
\end{align}
where $\bq\in\R^p$ is a centered random vector with rank $1\leq r_{\bq}=\rank(\bSigma_{\bq})\leq r_{y}$ and range $\mR(\bSigma_{\bq})\subseteq\mB_{y}$, $\balpha\in\R^p$ a vector of latent coefficients, $\eps\in\R$ an independent random variable that is centered, $\signot\geq0$ a noise parameter, $\be\in\R^p$ an independent random vector that is centered with full range $\mR(\bSigma_{\be})=\mB_{y}$ and normalized with $\lambda_{1}(\bSigma_{\be})=1$. Again, the latter display holds without loss of generality since one recovers the latent factor linear model in Equation~\eqref{eq:x.y.lfm} when $\bx_{y}=\bx$ and $\bx_{y^\bot}=\bzero_{p}$.

Under extended latent factor models one finds moments $\bsigma_{\bx,y}=\bsigma_{\bq,y}$ and $\bSigma_{\bx}=\bSigma_{\bq}+\signot^2\bSigma_{\be}+\bSigma_{\bx^\bot}$ and it is still natural to assume that $\signot^2<\lambda_{r_{\bq}}(\bSigma_{\bq})$. Even when the noise level is small, no restriction is imposed on the covariance $\bSigma_{\bx^\bot}$ of the irrelevant features $\bx_{y^\bot}$ and the $r_{\bq}$-dimensional principal eigenspace of $\bSigma_{\bx}$ might be far from the oracle linear subspace $\mR(\bSigma_{\bq})\subseteq\mB_{y}$. Notice that our extended model coincides with the classical model if and only if the irrelevant features $\bx_{y^\bot}$ are trivially zero, meaning that all features $\bx=\bx_{y}$ are correlated with the response. In general, the presence of the irrelevant features complicates the analysis since the partition $\bx=\bx_{y}+\bx_{y^\bot}$ is unknown and the covariance $\bSigma_{\bx^\bot}$ of the irrelevant features $\bx_{y^\bot}$ is arbitrary.

Under the well-specified model in Equation~\eqref{eq:x.y.elfm} it is natural to consider the oracle projection of $\bbeta_{\ls}\in\mB_{y}$, that is to say, the vector $\bU_{\bq} \bbeta_{\ls}$ where $\bU_{\bq}$ is the orthogonal projection of $\R^p$ onto the oracle range $\mR(\bSigma_{\bq})$. We show in Lemma~\ref{lem:x.y.elfm} that this coincides with the solution of the population least-squares problem $\ls(\bx_{\bq},y)$ computed from the oracle projection of the features $\bx_{\bq}:=\bU_{\bq}\bx_{y}+\bU_{\bq}\bx_{y^\bot}=\bq+\signot\bU_{\bq}\be+\bzero_{p}$. We also show that, with a signal-to-noise ratio $\lambda_{r_{\bq}}(\bSigma_{\bq})/\signot^2 > 2$, the projected vector $\bU_{\bq}\bbeta_{\ls}$ has an approximation error for the solution $\balpha_{\ls}:=\bSigma_{\bq}^\dagger\bsigma_{\bq,y}$ of the latent population least-squares problem $\ls(\bq,y)$ that is proportional to the inverse signal-to-noise ratio.

\subsection{Oracle Parsimonious Linear Reduction}\label{sec:plr}
We strive for a notion of parsimonious linear reduction that can be defined for general data generating processes on the relevant population pair $(\bx_{y},y)$ under Assumption~\ref{ass:x.y.model.free.2nd}. We propose an inductive definition that exploits the gradient of the least-squares functional $\bbeta\mapsto\ell_{\bx_{y},y}(\bbeta):=\E(y-\bx_{y}^t\bbeta)^2$ defined for all $\bbeta\in\R^p$. This gradient is $\bbeta\mapsto\nabla_{\bbeta}\ell_{\bx_{y},y}(\bbeta):=2\bSigma_{\bx_{y}}\bbeta-2\bsigma_{\bx_{y},y}\in\R^p$ and, by definition of relevant subspace in Equation~\eqref{eq:rel.sub}, one finds $\nabla_{\bbeta}\ell_{\bx_{y},y}(\bbeta)\in\mR(\bSigma_{\bx_{y}})=\mB_{y}$ for all $\bbeta\in\R^p$. Starting with the trivial direction $\bw_{0}:=\bzero_{p}\in\mB_{y}$, the trivial subspace $\mB_{0}:=\{\bzero_{p}\}\subseteq\mB_{y}$ and the trivial parameter $\bbeta_{0}:=\bzero_{p}\in\mB_{0}$, we look for the direction $\bw_{1}\in\mB_{y}$ of steepest descent for the functional $\ell_{\bx_{y},y}(\cdot)$ at the point $\bbeta_{0}$. This corresponds to the negative gradient $\bw_{1}:=-\nabla_{\bbeta}\ell_{\bx_{y},y}(\bbeta_{0})$ and we can define the linear subspace $\mB_{1}:=\spa\{\bw_{0},\bw_{1}\}$ and least-squares solution $\bbeta_{1}:=\argmin_{\bbeta\in\mB_{1}}\ \ell_{\bx_{y},y}(\bbeta)$. By iterating such procedure for all $1\leq s\leq r_{y}$, we formally define
\begin{align}\label{eq:x.rel.y.B.s.grad}
    \bw_{s}:=-\nabla_{\bbeta}\ell_{\bx_{y},y}(\bbeta_{s-1}),\quad \mB_{s}:=\spa\{\bw_{0},\ldots,\bw_{s}\},\quad \bbeta_{s}:=\argmin_{\bbeta\in\mB_{s}}\ \ell_{\bx_{y},y}(\bbeta). 
\end{align}
From the numerical theory established by \cite{hestenes1952met}, \cite{Allwright1976} and \cite{hanke1995conj} on conjugate gradient methods, the aforementioned linear subspaces span the population Krylov subspaces
\begin{align}\label{eq:x.rel.y.B.s.krylov}
    \mB_{s} = \spa\big\{\bsigma_{\bx_{y},y},\ldots,\bSigma_{\bx_{y}}^{s-1}\bsigma_{\bx_{y},y}\big\} =: \mK_{s}(\bx_{y},y),\quad 1\leq s\leq r_{y},
\end{align}
so that $\mB_{1}=\spa\{\bsigma_{\bx_{y},y}\}$ is also the direction of maximal correlation between the relevant features $\bx_{y}$ and the response $y$, and $\bsigma_{\bx_{y},y}\neq\bzero_{p}$ by Assumption~\ref{ass:x.y.model.free.2nd}. With $m_{y}:=\dim(\mB_{r_{y}})$ and $d_{y}:=\deg(\bSigma_{\bx_{y}})$ the number of unique non-zero eigenvalues of the covariance matrix $\bSigma_{\bx_{y}}$, we find the relationship $1\leq m_{y}\leq d_{y}\leq r_{y}$. That is to say, the sequence of linear subspaces $\mB_{1}\subsetneq\cdots\subsetneq\mB_{m_{y}}$ is strictly monotone. For all $1\leq s\leq m_{y}$, we define $\mB_{s}$ as the \textit{$s$-parsimonious linear reduction} of the relevant features. We denote $\bx_{s}$ the orthogonal projection of the relevant features $\bx_{y}$ onto $\mB_{s}$ and $\bbeta_{s}$ the solution of the population least-squares problem $\ls(\bx_{s},y)$. We call \textit{best parsimonious linear reduction} of the relevant features the linear subspace $\mB_{s_{0}}$ where
\begin{align}\label{eq:x.rel.y.B.s.star}
    s_{0} := \min \Big\{\argmin_{1\leq s\leq m_{y}}\ \E(y-\bx_{s}^t\bbeta_{s})^2 \Big\}.
\end{align}
The $\argmin$ in the above display is a set that might contain multiples solutions, thus $s_{0}$ is the smallest dimension for which the minimal linear least-squares residual is achieved. With $\bx_{s_{0}}$ the orthogonal projection of $\bx_{y}$ onto $\mB_{s_{0}}$, the \textit{best parsimonious parameter} $\bbeta_{s_{0}}\in\mB_{s_{0}}$ is the minimum-$L^2$-norm solution of the population least-squares problem $\ls(\bx_{s_{0}},y)$. 

Notice that the linear subspaces $\mB_{s}$ are not necessarily optimal in the least-squares sense. In fact, despite $\mB_{1}=\spa\{\bsigma_{\bx_{y},y}\}$ being the direction of maximal correlation it is easy to check that the optimal $1$-dimensional linear subspace of $\mB_{y}$ where the smallest least-squares residual is attained is $\spa\{\bbeta_{\ls}\}$. However, the latter trivially contains the population least-squares solution $\bbeta_{\ls}$ and nothing meaningful can be said about this projection.

To validate our proposal, we compare our construction of $s$-parsimonious linear reductions with the oracle linear subspace provided by the extended latent factor model in Equation~\eqref{eq:x.y.elfm}. Recall that this model assumes $y=\bq^t\balpha+\eps$ and $\bx=\bx_{y}+\bx_{y^\bot}=\bq+\signot\be+\bx_{y^\bot}$ with moments $\bsigma_{\bx,y}=\bsigma_{\bq,y}$ and $\bSigma_{\bx}=\bSigma_{\bq}+\signot^2\bSigma_{\be}+\bSigma_{\bx^\bot}$. The population Krylov spaces in Equation~\eqref{eq:x.rel.y.B.s.krylov} become $\mB_{s}=\mK_{s}(\bSigma_{\bq}+\signot^2\bSigma_{\be},\bsigma_{\bq,y})$ and are perturbed versions of the latent $\mK_{s}(\bSigma_{\bq},\bsigma_{\bq,y})$. Without loss of generality the latent range spans the whole latent Krylov space $\mR(\bSigma_{\bq})=\mK_{r_{\bq}}(\bSigma_{\bq},\bsigma_{\bq,y})$. Now consider the $r_{\bq}$-parsimonious parameter $\bbeta_{r_{\bq}}\in\mB_{r_{\bq}}$ in Equation~\eqref{eq:x.rel.y.B.s.krylov} where $r_{\bq}$ is the rank of the latent features $\bq$. We show in Lemma~\ref{lem:x.y.plr.elfm} that for a signal-to-noise ratio $\lambda_{r_{\bq}}(\bSigma_{\bq})/\signot^2 > 4$ the $r_{\bq}$-parsimonious parameter $\bbeta_{r_{\bq}}$ has an approximation error for the solution $\balpha_{\ls}$ of the latent population least-squares problem $\ls(\bq,y)$ that is proportional to the inverse signal-to-noise ratio. Up to a constant, this is the same approximation error we found in Lemma~\ref{lem:x.y.elfm} for the oracle projection $\bU_{\bq}\bbeta_{\ls}$ we discussed at the end of the previous section. Lastly, when the noise level $\signot\geq0$ is sufficiently small, the vector $\bbeta_{r_{\bq}}\in\mB_{r_{\bq}}$ is the best approximation of $\balpha_{\ls}\in\mR(\bSigma_{\bq})$ among all $\bbeta_{s}\in\mB_{s}$ over $1\leq s\leq m_{y}$. When this is true, we show in Lemma~\ref{lem:x.rel.y.B.s.latent} that the minimal dimension $s_{0}$ in Equation~\eqref{eq:x.rel.y.B.s.star} is at most the rank $r_{\bq}$ of the latent features $\bq$.

\section{Partial Least Squares}\label{sec:pls}
In this section we investigate the performance of the PLS algorithm in estimating the best parsimonious parameter $\bbeta_{s_{0}}\in\mB_{s_{0}}$ under a model-free setting or the oracle latent parameter $\balpha_{\ls}\in\mR(\bSigma_{\bq})$ under an extended latent factor model.

\subsection{Population Partial Least Squares}\label{sec:pls:pop}
The population PLS algorithm $\pls(\bx,y) := \pls(\bSigma_{\bx},\bsigma_{\bx,y})$ only depends on the moments of the population pair $(\bx,y)$ and does not have any knowledge on the partition of $\bx$ into relevant $\bx_{y}$ and irrelevant $\bx_{y^\bot}$ from Equation~\eqref{eq:x.rel.x.irr}. Following \cite{Wold66non} and \cite{helland1990pls}, the population PLS algorithm computes the minimum-$L^2$-norm least-squares solutions on the population Krylov subspaces  $\bbeta_{\pls,s}\in\mK_{s}(\bx,y)=\spa\{\bsigma_{\bx,y},\ldots,\bSigma_{\bx}^{s-1}\bsigma_{\bx,y}\}\subseteq\mR(\bSigma_{\bx})$ for all $1\leq s\leq p$. With $m_{\bx}:=\dim(\mK_{p}(\bx,y))$ and $d_{\bx}=\deg(\bSigma_{\bx})$ the number of unique non-zero eigenvalues of $\bSigma_{\bx}$, we find $1\leq m_{\bx}\leq d_{\bx}\leq r_{\bx}$. We prove the following adaptivity result in Section~\ref{app:proof:pls}.

\begin{lemma}\label{lem:x.y.pls.pop}
    Let $(\bx,y)\in\R^p\times\R$ satisfy Assumption~\ref{ass:x.y.model.free.2nd}. The population PLS algorithm is adaptive in the sense that $\pls(\bx,y)=\pls(\bx_{y},y)$. That is to say, $\mK_{s}(\bx,y)=\mK_{s}(\bx_{y},y)$ for all $1\leq s\leq r_{\bx}$.  
\end{lemma}

An immediate consequence of the above result is that the population PLS algorithm $\pls(\bx,y)$ recovers exactly the $s$-parsimonious linear subspaces $\mK_{s}(\bx,y)=\mB_{s}$ in Equation~\eqref{eq:x.rel.y.B.s.krylov} and the corresponding $s$-parsimonious parameters $\bbeta_{\pls,s}=\bbeta_{s}\in\mB_{s}$ solving the population least-squares problem $\ls(\bx_{s},y)$. We prove the following in Section~\ref{app:proof:pls}.

\begin{theorem}\label{thm:x.y.pls.pop}
    Let $(\bx,y)\in\R^p\times\R$ satisfy Assumption~\ref{ass:x.y.model.free.2nd}. Let $\bbeta_{\pls,s}$ be the coefficients computed by the population PLS algorithm $\pls(\bx,y)$ for all $1\leq s\leq r_{\bx}$. With $\bbeta_{s_{0}}\in\mB_{s_{0}}$ the best parsimonious parameter induced by Equation~\eqref{eq:x.rel.y.B.s.star}, then
    \begin{align*}
        \frac{\|\bbeta_{\pls,s}-\bbeta_{s_{0}}\|_{2}}{\|\bbeta_{s_{0}}\|_{2}} &\leq \sqrt{s_{0}-s},
    \end{align*}
    for all $1\leq s\leq s_{0}$.
\end{theorem}

In the next section we study the sample PLS algorithm and show that its parameters $\wh{\bbeta}_{\pls,s}$ computed with $1\leq s\leq s_{0}$ degrees-of-freedom converge in probability to the corresponding population parameters $\bbeta_{\pls,s}$. The above result thus quantifies the bias of the sample PLS solutions $\wh{\bbeta}_{\pls,s}$ with respect to the best parsimonious parameter $\bbeta_{s_{0}}$. This holds for all choices of degrees-of-freedom and does not rely on heuristic stopping rules. In particular, it shows that the sample PLS solution $\wh{\bbeta}_{\pls,s_{0}}$ using exactly $s_{0}$ degrees-of-freedom is an unbiased estimator of the best parsimonious parameter. This is true in the most general model-free setting where the dependence between the features and the response is arbitrary. 

Under the extended latent factor model in Equation~\eqref{eq:x.y.elfm} and a signal-to-noise ratio $\lambda_{r_{\bq}}(\bSigma_{\bq})/\signot^2>4$, we show in Section~\ref{app:proof:pls} the following.

\begin{theorem}\label{thm:x.y.pls.pop.elfm}
    Let $(\bx,y)\in\R^p\times\R$ satisfy Assumption~\ref{ass:x.y.model.free.2nd}. Let $\bbeta_{\pls,s}$ be the coefficients computed by the population PLS algorithm $\pls(\bx,y)$ for all $1\leq s\leq r_{\bx}$. Under the extended latent factor model in Equation~\eqref{eq:x.y.elfm} let $\balpha_{\ls}\in\mR(\bSigma_{\bq})$ be the minimum-$L^2$-norm solution of the latent population least-squares problem $\ls(\bq,y)$. With $r_{\bq}=\rank(\bSigma_{\bq})$ the rank of the latent features and some constant $C_{r_{\bq}}\geq1$, if $\signot^2<\lambda_{r_{\bq}}(\bSigma_{\bq})/2\{C_{r_{\bq}}+1\}$, then
    \begin{align*}
        \frac{\|\bbeta_{\pls,s}-\balpha_{\ls}\|_{2}}{\|\balpha_{\ls}\|_{2}} &\leq \frac{7}{2}\sqrt{r_{\bq}-s} + 5\ \{C_{r_{\bq}}+1\}\ \frac{\signot^2}{\lambda_{r_{\bq}}(\bSigma_{\bq})},
    \end{align*}
    for all $1\leq s\leq r_{\bq}$.
\end{theorem}

The above result measures the bias of the sample PLS solutions $\wh{\bbeta}_{\pls,s}$ using $1\leq s\leq r_{\bq}$ degrees-of-freedom with respect to the oracle latent parameter $\balpha_{\ls}$. The PLS method does not have any prior knowledge on the latent features nor on the partition of the features into relevant and irrelevant parts. The above theorem shows that the sample PLS solution $\wh{\bbeta}_{\pls,r_{\bq}}$ using exactly $r_{\bq}$ degrees-of-freedom attains a bias that is equal, up to the factor $C_{r_{\bq}}+1\geq2$, to the oracle approximation error obtained in Lemma~\ref{lem:x.y.elfm} for the oracle projection $\bU_{\bq}\bbeta_{\ls}$ of the population least-squares solution $\bbeta_{\ls}$ onto the oracle linear subspace $\mR(\bSigma_{\bq})$ which is the range of the latent features $\bq$. 

Under the same setting, with a larger signal-to-noise ratio $\lambda_{r_{\bq}}(\bSigma_{\bq})/\signot^2>2\tau$ for some $\tau\geq8$, we establish the following stopping rule for the population PLS algorithm. For a proof see Section~\ref{app:proof:pls}.

\begin{theorem}\label{thm:x.y.pls.pop.elfm.stop}
   Under the assumptions of Theorem~\ref{thm:x.y.pls.pop.elfm}, let $\bU_{s}$ be the orthogonal projection of $\R^p$ onto the population Krylov space $\mK_{s}(\bSigma_{\bx},\bsigma_{\bx,y})$ for all $1\leq s\leq p$. Furthermore, assume that $\signot^2<\lambda_{r_{\bq}}(\bSigma_{\bq})/\tau\{C_{r_{\bq}}+1\}$ for some $\tau\geq8$. Then the population early-stopping dimension
   \begin{align*}
       m_{\bq} := \min\left\{1\leq s\leq p-1: \frac{\kappa_{2}(\bU_{s+1}\bSigma_{\bx}\bU_{s+1})}{\kappa_{2}(\bU_{s}\bSigma_{\bx}\bU_{s})} > \tau-2 \right\}
   \end{align*}
   satisfies $m_{\bq}\leq r_{\bq}$.
\end{theorem}

The above result establishes a stopping rule for the population PLS algorithm. In particular, it shows that one should consider the condition numbers $\kappa_{s}:=\kappa_{2}(\bU_{s}\bSigma_{\bx}\bU_{s})$ computed for all degrees-of-freedom $1\leq s\leq r_{\bx}$ and take the first one for which $\kappa_{s+1}/\kappa_{s}>\tau-2$. Here the quantity $\tau\geq8$ is meant to be known, but one can replace this with the agnostic $\kappa_{s+1}/\kappa_{s}>6$ corresponding to $\tau=8$ instead. Notice that we are not imposing any additional restriction on the decay or separation of the eigenvalues of the covariance $\bSigma_{\bx}$ of the features. Stronger assumptions such as polynomial or exponential decay would allow for larger gaps in the ratios of conditioning numbers.

\subsection{Sample Partial Least Squares}\label{sec:pls:sam}
Consider a dataset $(\bX,\by)\in\R^{n\times p}\times\R^n$ consisting of $n\geq1$ i.i.d. realizations $(\bx_{i},y_{i})$ of the same population pair $(\bx,y)\in\R^p\times\R$ under Assumption~\ref{ass:x.y.model.free.2nd}. In this section we investigate the performance of the sample PLS algorithm $\wh{\pls}(\bx,y) := \pls(\wh{\bSigma}_{\bx},\wh{\bsigma}_{\bx,y})$ that depends only on the sample moments $\wh{\bSigma}_{\bx}:=n^{-1}\bX^t\bX$ and $\wh{\bsigma}_{\bx,y}:=n^{-1}\bX^t\by$ estimated from the dataset $(\bX,\by)$. Following \cite{Wold66non} and \cite{helland1990pls}, the sample PLS algorithm computes the minimum-$L^2$-norm least-squares solutions on the sample Krylov subspaces $\wh{\bbeta}_{\pls,s}\in\wh{\mK}_{s}(\bx,y)=\spa\{\wh{\bsigma}_{\bx,y},\ldots,\wh{\bSigma}_{\bx}^{s-1}\wh{\bsigma}_{\bx,y}\}$ for all $1\leq s\leq p$. With $\wh{m}_{\bx}:=\dim(\wh{\mK}_{p}(\bx,y))$ and $\wh{d}_{\bx}=\deg(\wh{\bSigma}_{\bx})$ the number of unique non-zero eigenvalues of $\wh{\bSigma}_{\bx}$, we find $1\leq \wh{m}_{\bx}\leq \wh{d}_{\bx}\leq \wh{r}_{\bx} := \rank(\wh{\bSigma}_{\bx})$. In what follows, $\|\cdot\|_{op}$ is the operator norm for matrices. We denote
\begin{align} \label{eq:x.y.eps.hat}
    \wh{\eps}(\bx,y) := \frac{\|\wh{\bSigma}_{\bx}-\bSigma_{\bx}\|_{op}}{\|\bSigma_{\bx}\|_{op}} \vee \frac{\|\wh{\bsigma}_{\bx,y}-\bsigma_{\bx,y}\|_{2}}{\|\bsigma_{\bx,y}\|_{2}},
\end{align}
the size of the perturbation between the sample moments $\wh{\bSigma}_{\bx},\wh{\bsigma}_{\bx,y}$ and the population moments $\bSigma_{\bx},\bsigma_{\bx,y}$. 

\begin{assumption}[Model-Free, 4th moments]\label{ass:x.y.model.free.4th}
    Let $(\bx,y)\in\R^p\times\R$ satisfy Assumption~\ref{ass:x.y.model.free.2nd}. The response $y$ and the projected features $\bx_{\wt{\mB}}=\bU_{\wt{\mB}}\ \bx$, for any linear subspace $\wt{\mB}\subseteq\mR(\bSigma_{\bx})$, have finite moment-ratios
    \begin{align*}
        L_{y} :=\frac{\E(y^4)^{\frac{1}{4}}}{\E(y^2)^{\frac{1}{2}}},\quad
        L_{\wt{\mB}} := \frac{\E(\|\bx_{\wt{\mB}}\|_{2}^4)^{\frac{1}{4}}}{\E(\|\bx_{\wt{\mB}}\|_{2}^2)^{\frac{1}{2}}},
    \end{align*}
    with the convention that $L_{\wt{\mB}}$ is set to one if the denominator is zero.
\end{assumption}

Let $(\bx,y)\in\R^p\times\R$ satisfy Assumption~\ref{ass:x.y.model.free.4th}. From now on, we are interested in the geometrical properties of the projected features $\bx_{\wt{\mB}}$ where the linear subspace $\wt{\mB}$ is either $\mB_{s}$, $\mB_{s}^\bot$ or $\mB_{y}^\bot$ for some fixed $1\leq s\leq r_{y}$ as in Equation~\eqref{eq:x.rel.y.B.s.krylov}. For any such $\wt{\mB}$, we denote
\begin{align}\label{eq:x.proj.geom}
    r_{\wt{\mB}} := \rank(\bSigma_{\bx_{\wt{\mB}}}),\quad 
    \rho_{\wt{\mB}} := \frac{\E(\|\bx_{\wt{\mB}}\|_{2}^2)}{\|\bSigma_{\bx_{\wt{\mB}}}\|_{op}},\quad 
    \rho_{\wt{\mB},n} := \frac{\E\big(\max_{1\leq i\leq n} \|\bx_{\wt{\mB},i}\|_{2}^2\big)}{\|\bSigma_{\bx_{\wt{\mB}}}\|_{op}}.
\end{align}
The rank $r_{\wt{\mB}}$ is the dimension of the span of the support of $\bx_{\wt{\mB}}$. The effective rank $\rho_{\wt{\mB}} \leq r_{\wt{\mB}}$ can be rewritten as the weighted average $\Tr(\bSigma_{\bx_{\wt{\mB}}})/\|\bSigma_{\bx_{\wt{\mB}}}\|_{op}$ and measures the interplay between dimension and variation. The uniform effective rank $\rho_{\wt{\mB},n}$ accounts for the variability of a sample of i.i.d. realizations of $\bx_{\wt{\mB}}$. We select the linear subspace among $\mB_{s}$, $\mB_{s}^\bot$ or $\mB_{y}^\bot$ corresponding to the largest variation
\begin{align} \label{eq:x.y.B.star}
    \wt{\mB}_{s} := \argmax \left\{L_{\wt{\mB}}\ \|\bSigma_{\bx_{\wt{\mB}}}\|_{op}\ \rho_{\wt{\mB},n} : \wt{\mB}\in\{\mB_{s}, \mB_{s}^\bot, \mB_{y}^\bot\} \right\}
\end{align}
and define the sequence
\begin{align} \label{eq:x.y.delta.star.n}
    \delta_{\wt{\mB}_{s},n} := \sqrt{\frac{\rho_{\wt{\mB}_{s},n} \log r_{\bx}}{n}},
\end{align}
summarizing the intrinsic geometrical complexity. 

\cite{finocchio2025mod} defined a notion of stability for the population PLS algorithm $\pls(\bx,y)$ computed from a pair $(\bx,y)$ under Assumption~\ref{ass:x.y.model.free.4th}. In what follows, we denote $\wt{C}_{s}\geq1$ such stability constants and let
\begin{align*}
    \wt{M}_{s} := 2 \cdot \kappa_{2}(\bSigma_{\bx_{s}}) \cdot \{4\ \wt{C}_{s} + 1\} \cdot \left\{\frac{\|\bSigma_{\bx}\|_{op}}{\|\bSigma_{\bx_{s}}\|_{op}} \vee \frac{\|\bsigma_{\bx,y}\|_{2}}{\|\bsigma_{\bx_{s},y}\|_{2}}\right\}.
\end{align*} 
for all $1\leq s\leq m_{y}$. 

\begin{assumption}[Sample PLS Algorithm]\label{ass:x.y.pls.sam}
    We assume that:
    \begin{enumerate}[label=(\roman*),itemsep=0.25em,topsep=0.25em]
        \item the sample PLS algorithm is compatible with the population PLS algorithm, in the sense that $\dim(\wh{\mK}_{p}(\bx,y))\geq\dim(\mK_{p}(\bx,y))$, \label{ass:x.y.pls.sam.dof}
        \item with $\wt{\mB}_{s}$ the leading linear subspace among $\mB_{s}, \mB_{s}^\bot, \mB_{y}^\bot$ in the sense of Equation~\eqref{eq:x.y.B.star}, $\delta_{\wt{\mB}_{s},n}$ the corresponding complexity in Equation~\eqref{eq:x.y.delta.star.n}, some absolute constant $C\geq1$,
        \begin{align*}
            K_{\wt{\mB}_{s}} &:= 99 C L_{y} L_{\wt{\mB}_{s}} \left\{\frac{\sigma_{y} \|\bSigma_{\bx_{\wt{\mB}_{s}}}\|_{op}^\frac{1}{2}}{\|\bsigma_{\bx,y}\|_{2}} \vee \frac{\|\bSigma_{\bx_{\wt{\mB}_{s}}}\|_{op}}{\|\bSigma_{\bx}\|_{op}} \right\},
        \end{align*}
        it holds
        \begin{align*}
            \delta_{\wt{\mB}_{s},n} \xrightarrow{n\to\infty} 0,\quad \nu_{\wt{\mB}_{s},n} := \wt{M}_{s} K_{\wt{\mB}_{s}} \delta_{\wt{\mB}_{s},n} < \frac{1}{2}.
        \end{align*} \label{ass:x.y.pls.sam.n.large}
    \end{enumerate}
\end{assumption}

The next result, which we provide without proof, follows from Theorem~\ref{thm:x.y.pls.pop} and Theorem~2.14 by \cite{finocchio2025mod}.

\begin{theorem} \label{thm:x.y.pls.sam}
    Let $(\bx,y)\in\R^p\times\R$ satisfy Assumption~\ref{ass:x.y.model.free.4th}. Let $(\bX,\by)\in\R^{n\times p}\times\R^n$ be a dataset of i.i.d. realizations of $(\bx,y)$ and Assumption~\ref{ass:x.y.pls.sam} hold. Let $\bbeta_{s_{0}}\in\mB_{s_{0}}$ be the best parsimonious parameter induced by Equation~\eqref{eq:x.rel.y.B.s.star} and, for all $1\leq s\leq s_{0}$, let $\wh{\bbeta}_{\pls,s}$ be the sample PLS coefficients computed from $\wh{\pls}(\bx,y)$. Then, for any $\nu_{\wt{\mB}_{s},n} < \nu_{s,n} < \frac{1}{2}$, the size of the perturbation $\wh{\eps} = \wh{\eps}(\bx,y)$ in Equation~\eqref{eq:x.y.eps.hat} satisfies $\wh{\eps} \leq K_{\wt{\mB}_{s}}\ \nu_{s,n}^{-1}\ \delta_{\wt{\mB}_{s},n}$ with probability at least $1-2\nu_{s,n}$. On this event, one has
    \begin{align*}
        \frac{\|\wh{\bbeta}_{\pls,s} - \bbeta_{s_{0}}\|_{2}}{\|\bbeta_{s_{0}}\|_{2}} &\leq \sqrt{s_0-s} + \frac{5}{2} \wt{M}_{s} K_{\wt{\mB}_{s}} \sqrt{\frac{\rho_{\wt{\mB}_{s},n} \log r_{\bx}}{n\nu_{s,n}^2}}.
    \end{align*}
\end{theorem}

Under Assumption~\ref{ass:x.y.pls.sam} it is always possible to select $\nu_{s,n}\to0$ arbitrarily slow, when $n\to\infty$, so that $\nu_{s,n}^{-1}\ \delta_{\wt{\mB}_{s},n}\to0$ as well. The above result then shows that the sample PLS solution $\wh{\bbeta}_{\pls,s_{0}}$ using exactly $s_{0}$ degrees-of-freedom is an unbiased estimator fo the best parsimonious parameter $\bbeta_{s_{0}}$. Notice that the sample PLS algorithm only depends on the observed data $(\bX,\by)$ and has no knowledge of the factorization of the features into relevant and irrelevant parts. To the best of our knowledge, the above result is the first to provide a transparent characterization of the PLS method under random design in the model-free setting from Assumption~\ref{ass:x.y.model.free.4th}. For a discussion on the optimality on the above convergence rates, we refer to Remark~2.17 by \cite{finocchio2025mod}. Interestingly, we also generalize a result established by \cite{Chun2010} showing that PLS estimators are inconsistent when $p/n\to c>0$. We only require the uniform effective rank to be sufficiently small that $\rho_{n}/n\to0$ in our Assumption~\ref{ass:x.y.pls.sam}. All the results obtained in this section can be easily extended to the setting where the observed data is not i.i.d. as in the work by \cite{singer2016partial}. They assumed some underlying sample $(\bX,\by)\in\R^{n\times p}\times\R^n$ of i.i.d. observations but one only observes $\wt{\bX}=\bSigma_{n}^{1/2}\bX$ and $\wt{\by}=\bSigma_{n}^{1/2}\by$ for some unknown temporal covariance matrix $\bSigma_{n}\in\R_{\succ0}^{n\times n}$. Under the assumption that a consistent estimator $\wh{\bSigma}_{n}\in\R_{\succ0}^{n\times n}$ for the temporal covariance is available, then the convergence rates of the sample PLS solutions $\wh{\bbeta}_{\pls,s}$ computed from the normalized dataset $(\wh{\bSigma}_{n}^{-1/2}\wt{\bX},\wh{\bSigma}_{n}^{-1/2}\wt{\by})$ has an additional term that is proportional to $\|\wh{\bSigma}_{n}-\bSigma_{n}\|_{op}$. We refer to Section~4 by \cite{singer2016partial} for more details.

Under the extended latent factor model in Equation~\eqref{eq:x.y.elfm} and a signal-to-noise ratio $\lambda_{r_{\bq}}(\bSigma_{\bq})/\signot^2>4$, the next result follows from Theorem~\ref{thm:x.y.pls.pop.elfm} and Theorem~\ref{thm:x.y.pls.sam}. This result is provided without proof.

\begin{theorem}\label{thm:x.y.pls.sam.elfm}
    Under the assumptions of Theorem~\ref{thm:x.y.pls.pop.elfm} and Theorem~\ref{thm:x.y.pls.sam}, let $\balpha_{\ls}\in\mR(\bSigma_{\bq})$ be the minimum-$L^2$-norm solution of the latent population least-squares problem $\ls(\bq,y)$ from the extended latent factor model in Equation~\eqref{eq:x.y.elfm}. On the same event of probability at least $1-2\nu_{s,n}$, one has
    \begin{align*}
        \frac{\|\wh{\bbeta}_{\pls,s}-\balpha_{\ls}\|_{2}}{\|\balpha_{\ls}\|_{2}} &\leq \frac{7}{2}\sqrt{r_{\bq}-s} + 5 \{C_{r_{\bq}}+1\} \frac{\signot^2}{\lambda_{r_{\bq}}(\bSigma_{\bq})} + \frac{5}{2} \wt{M}_{s} K_{\wt{\mB}_{s}} \sqrt{\frac{\rho_{\wt{\mB}_{s},n} \log r_{\bx}}{n\nu_{s,n}^2}}.
    \end{align*}
\end{theorem}

Our novel proving strategy allows to comprehensively study the convergence rates of the sample PLS solutions $\wh{\bbeta}_{\pls,s}$ for all degrees-of-freedom $1\leq s\leq r_{\bq}$ with respect to the oracle latent solution $\balpha_{\ls}$. Although \cite{singer2016partial} studied convergence rates for PLS estimators, we improve on their results in different notable directions. First, they only considered a classical latent factor model as in Equation~\eqref{eq:x.y.lfm} without the possibility of irrelevant features $\bx_{y\bot}$. Second, they only considered the convergence of the PLS estimator to its population counterpart $\bbeta_{\pls,s}$ but not in terms of oracle latent solution $\balpha_{\ls}$ in Equation~\eqref{eq:x.y.elfm}. Third, they only provide bounds for the parameter $1\leq\widehat{s}\leq p$ resulting from the heuristic stopping rule of \cite{nemirovskii1986reg} instead of any number of degrees-of-freedom $1\leq s\leq r_{\bq}$. 

Under the same setting and a signal-to-noise ratio $\lambda_{r_{\bq}}(\bSigma_{\bq})/\signot^2>2\tau$ for some $\tau\geq8$, we prove the next result in Section~\ref{app:proof:pls} using Theorem~\ref{thm:x.y.pls.pop.elfm.stop} and Theorem~\ref{thm:x.y.pls.sam.elfm}.

\begin{theorem}\label{thm:x.y.pls.sam.elfm.stop}
   Under the assumptions of Theorem~\ref{thm:x.y.pls.pop.elfm.stop} and Theorem~\ref{thm:x.y.pls.sam.elfm}, let $\wh{\bU}_{s}$ be the orthogonal projection of $\R^p$ onto the sample Krylov space $\mK_{s}(\wh{\bSigma}_{\bx},\wh{\bsigma}_{\bx,y})$ for all $1\leq s\leq p$. Furthermore, assume that
   \begin{align*}
       K_{\wt{\mB}_{r_{\bq}+1}}\ \nu_{r_{\bq}+1,n}^{-1}\ \delta_{\wt{\mB}_{r_{\bq}+1},n}  < \frac{\lambda_{r_{\bq}}(\bSigma_{\bq})}{6\tau \{\|\bSigma_{\bx}\|_{op}\vee\|\bsigma_{\bx,y}\|_{2}\}\ }.
   \end{align*}
   Then, for any $\nu_{\wt{\mB}_{r_{\bq}+1},n} < \nu_{r_{\bq}+1,n} < \frac{1}{2}$, with probability at least $1-2\nu_{r_{\bq}+1,n}$ the sample early-stopped dimension
   \begin{align*}
       \wh{m}_{\bq} := \min\left\{1\leq s\leq p-1: \frac{\kappa_{2}(\wh{\bU}_{s+1}\wh{\bSigma}_{\bx}\wh{\bU}_{s+1})}{\kappa_{2}(\wh{\bU}_{s}\wh{\bSigma}_{\bx}\wh{\bU}_{s})} > \frac{2\tau-5}{4} \right\}
   \end{align*}
   satisfies $\wh{m}_{\bq}\leq r_{\bq}$.
\end{theorem}

The above result establishes a stopping rule for the sample PLS algorithm. The method relies on the sample condition numbers $\wh{\kappa}_{s}:=\kappa_{2}(\wh{\bU}_{s}\wh{\bSigma}_{\bx}\wh{\bU}_{s})$ computed for all degrees-of-freedom $1\leq s\leq r_{\bx}$. As mentioned earlier, the quantity $\tau\geq8$ is meant to be known, but one can replace $\wh{\kappa}_{s+1}/\wh{\kappa}_{s}>\{2\tau-5\}/4$ in the above display with the agnostic $\wh{\kappa}_{s+1}/\wh{\kappa}_{s}>11/4$ corresponding to $\tau=8$ instead. Stronger structural assumptions such as polynomial or exponential decay of the eigenvalues of the sample covariance $\wh{\bSigma}_{\bx}$ would allow for larger gaps in the ratios of conditioning numbers. The idea of monitoring the convergence of the PLS algorithm in terms of its empirical conditioning is not new. In their Section~4, \cite{blanchard2010} do this for the general class of kernel-PLS algorithms. Of course, the sample stopping rule in Theorem~\ref{thm:x.y.pls.sam.elfm.stop} is meant to be parsimonious rather than optimal. Lastly, we would like to point out that \cite{Kramer2011deg} have proposed a notion of degrees-of-freedom for PLS regression that is different from ours. For us, the degrees-of-freedom of the sample PLS solutions $\wh{\bbeta}_{\pls,s}$ are the dimensions of the corresponding sample Krylov spaces $\dim(\mK_{s}(\wh{\bSigma}_{\bx},\wh{\bsigma}_{\bx,y}))=\rank(\wh{\bU}_{s})$, so that $r_{\bq}$ is the oracle number of degrees-of-freedom. This essentially corresponds to Equation~(4) by \cite{Kramer2011deg} instead of their Definition~1 inspired by \cite{Efron2004est}. 

Consider a dataset $(\bX,\by)\in\R^{n\times p}\times\R^n$ consisting of $n\geq1$ i.i.d. realizations $(\bx_{i},y_{i})$ of the same population pair $(\bx,y)\in\R^p\times\R$ under Assumption~\ref{ass:x.y.model.free.4th}. Let $y_{n+1}\in\R$ be a new unobserved response value and $\bx_{n+1}\in\R^p$ a new observed feature vector following the same population distribution. Let $\wh{\bbeta}_{\pls,s}$ be the sample PLS solution with $1\leq s\leq p$ computed from the data $(\bX,\by)$ and independent of the new pair $(\bx_{n+1},y_{n+1})$. With $\bbeta_{s_{0}}$ the best parsimonious parameter, we define the risk and excess risk 
\begin{align*}
    R_{\bx,y}(\wh{\bbeta}_{\pls,s}) := \E_{(\bx,y)}\big(\{y-\bx^t\wh{\bbeta}_{\pls,s}\}^2|\bX,\by\big), \quad
    R_{\bx,y}^{(ex)}(\wh{\bbeta}_{\pls,s}) := R_{\bx,y}(\wh{\bbeta}_{\pls,s}) - R_{\bx,y}(\bbeta_{s_{0}}),
\end{align*}
where the expectation is taken with respect to the population pair $(\bx,y)$ and conditionally on the data $(\bX,\by)$. The next result follows from Theorem~\ref{thm:x.y.pls.sam} and Theorem~2.18 by \cite{finocchio2025mod} and is provided without proof.

\begin{theorem} \label{thm:x.y.pls.pred}
    Under the assumptions of Theorem~\ref{thm:x.y.pls.sam}, on the same event with probability at least $1-2\nu_{s,n}$, the excess-risk is
    \begin{align*}
        R_{\bx,y}^{(ex)}(\wh{\bbeta}_{\pls,s}) &= \big\|\wh{\bbeta}_{\pls,s}-\bbeta_{s_{0}}\big\|_{\bSigma_{\bx}}^2 - 2 \big\langle \wh{\bbeta}_{\pls,s}-\bbeta_{s_{0}},\ \bbeta_{\ls}-\bbeta_{s_{0}} \big\rangle_{\bSigma_{\bx}}.
    \end{align*}
\end{theorem}

The above result deals with the problem of best parsimonious linear prediction $\bx_{n+1}^t\bbeta_{s_{0}}$ of the new unobserved response $y_{n+1}$ regardless of the true dependence between the features and the response. It shows that the excess risk of the sample PLS solution $\wh{\bbeta}_{\pls,s}$ is proportional to $\|\wh{\bbeta}_{\pls,s}-\bbeta_{s_{0}}\|_{\bSigma_{\bx}}^2$ which is the square of the $\bSigma_{\bx}$-weighted convergence rate we found in Theorem~\ref{thm:x.y.pls.sam}. In particular, the sample PLS solution $\wh{\bbeta}_{\pls,s_{0}}$ using exactly $s_{0}$ degrees-of-freedom is unbiased and so $R_{\bx,y}^{(ex)}(\wh{\bbeta}_{\pls,s_{0}})\to0$ in probability when $n\to\infty$. Although the sample PLS predictor $\bx_{n+1}^t\wh{\bbeta}_{\pls,s_{0}}$ might far from any possibly overparametrized predictor $\wh{f}(\bx_{n+1})$ achieving optimal prediction risk for $y_{n+1}$, it is otherwise parsimonious and interpretable. 

It is immediate to formulate the corresponding version of the above theorem in the setting of extended latent factor model in Equation~\eqref{eq:x.y.elfm} under the assumptions of Theorem~\ref{thm:x.y.pls.sam.elfm}. In this setting, the oracle latent predictor for the new unobserved response $y_{n+1}$ is $\bx_{n+1}^t\balpha_{\ls}$ with $\balpha_{\ls}$ the oracle latent solution. Measuring the excess risk as $R_{\bx,y}^{(ex)}(\wh{\bbeta}_{\pls,s}) := R_{\bx,y}(\wh{\bbeta}_{\pls,s}) - R_{\bx,y}(\balpha_{\ls})$, one thus finds the finite-sample excess risk of the sample PLS solutions $\wh{\bbeta}_{\pls,s}$ to be proportional to $\|\wh{\bbeta}_{\pls,s}-\balpha_{\ls}\|_{\bSigma_{\bx}}^2$ which are the squares of the $\bSigma_{\bx}$-weighted convergence rates we found in Theorem~\ref{thm:x.y.pls.sam.elfm}. This improves upon previous works by \cite{bing2021pre}, who derived finite-sample prediction risk for projection methods only for classical latent models, and \cite{Cook2019par} who established the prediction risk of sample PLS only asymptotically.

\section{Numerical Studies} \label{sec:num}
We confirm our findings with empirical studies on both simulated and real datasets.

\subsection{Simulated Data}\label{sec:num-simul}
We simulate our dataset $(\bX,\by)$ according to the following scheme:
\begin{enumerate}[label=(\roman*)]
    \item we choose $n=2000$ and $p=200$; the number of relevant features is always $r_{y}=100$ and the true number of factors is always $r_{\bq}=25$;
    \item we draw the latent dataset $\bQ = (\bq_{1},\ldots,\bq_{n})^t \in\R^{n\times r_{\bq}}$ as
    \begin{align*}
        \bq_{i} \overset{ind}{\sim} \mN\Big(\bzero_{r_{\bq}}, \diag(\bsigma_{\bq}^2)\Big) \in \R^{r_{\bq}},\quad 5 = (\bsigma_{\bq})_{1} > \ldots > (\bsigma_{\bq})_{r_{\bq}} = 1 ;
    \end{align*}
    \item we draw the relevant dataset $\bQ_{y} = (\bq_{y,1},\ldots,\bq_{y,n})^t \in\R^{n\times r_{y}}$ as
    \begin{align*}
        \bq_{y,i} | \bq_{i} \overset{ind}{\sim} \mN\left( 
        \Big(\begin{matrix}
        \bq_{i} \\
        \bzero_{r_{y}-r_{\bq}}
    \end{matrix}\Big)
    , \diag(\bsigma_{0}^2) \right) \in \R^{r_{y}},\quad \signot = (\bsigma_{0})_{1} > \ldots > (\bsigma_{0})_{r_{y}} = 10^{-3}
    \end{align*}
    with $\signot=1$ for large noise level and induced signal-noise-ratio $(\bsigma_{\bq})_{r_{\bq}}^2/\signot^2=1$;
    \item we draw the irrelevant dataset $\bQ_{y^\bot} = (\bq_{y^\bot,1},\ldots,\bq_{y^\bot,n})^t \in\R^{n\times (p-r_{y})}$ as
    \begin{align*}
        \bq_{y^\bot,i} \overset{ind}{\sim} \mN\Big(\bzero_{p-r_{y}}, \diag(\bsigma_{y^\bot}^2) \Big) \in \R^{p-r_{y}},\quad (\bsigma_{y^\bot})_{r_{y^\bot}+1} = \ldots = (\bsigma_{y^\bot})_{p-r_{y}} = 0
    \end{align*}
    with largest eigenvalue $(\bsigma_{y^\bot})_{1}=2.5$ for strong irrelevant features and $(\bsigma_{y^\bot})_{1}=0.1$ for weak irrelevant features;
    \item with deterministic orthonormal matrix $\bU\in\R^{p\times p}$, we assemble the observed dataset $\bX = (\bx_{1},\ldots,\bx_{n})^t \in\R^{n\times p}$
    \begin{align*}
        \bX &= \left(\bQ_{y}\ |\ \bQ_{y^\bot}\right) \bU^t \in \R^{n\times p} ;
    \end{align*}
    \item with deterministic $\balpha_{0} = (1,2,\ldots,r_{\bq})^t\in\R^{r_{\bq}}$ we draw the observed response vector $\by = (y_{1},\ldots,y_{n})^t\in\R^n$ as
    \begin{align*}
        \quad y_{i} | \bq_{i} \overset{ind}{\sim} \mN\left(\bq_{i}^t\balpha_{0}, 1\right) \in \R ;
    \end{align*}
    \item with $\bP = \bU \bI_{r_{y},p} \bI_{r_{\bq},r_{y}} \in\R^{p\times r_{\bq}}$ we obtain the oracle coefficients $\bbeta_{0} = \bP\balpha_{0}\in\R^p$.
\end{enumerate}
\clearpage

\begin{figure}[ht]
\centering
\includegraphics[width=1\textwidth]{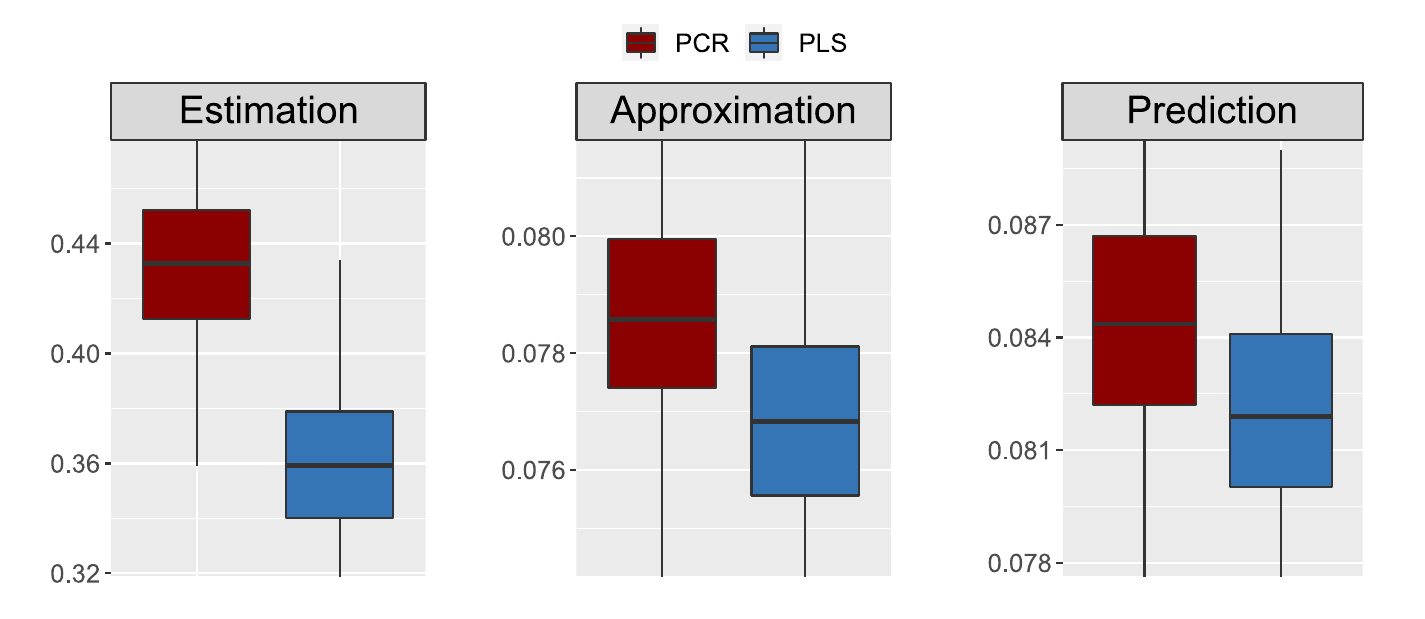}
\includegraphics[width=1\textwidth]{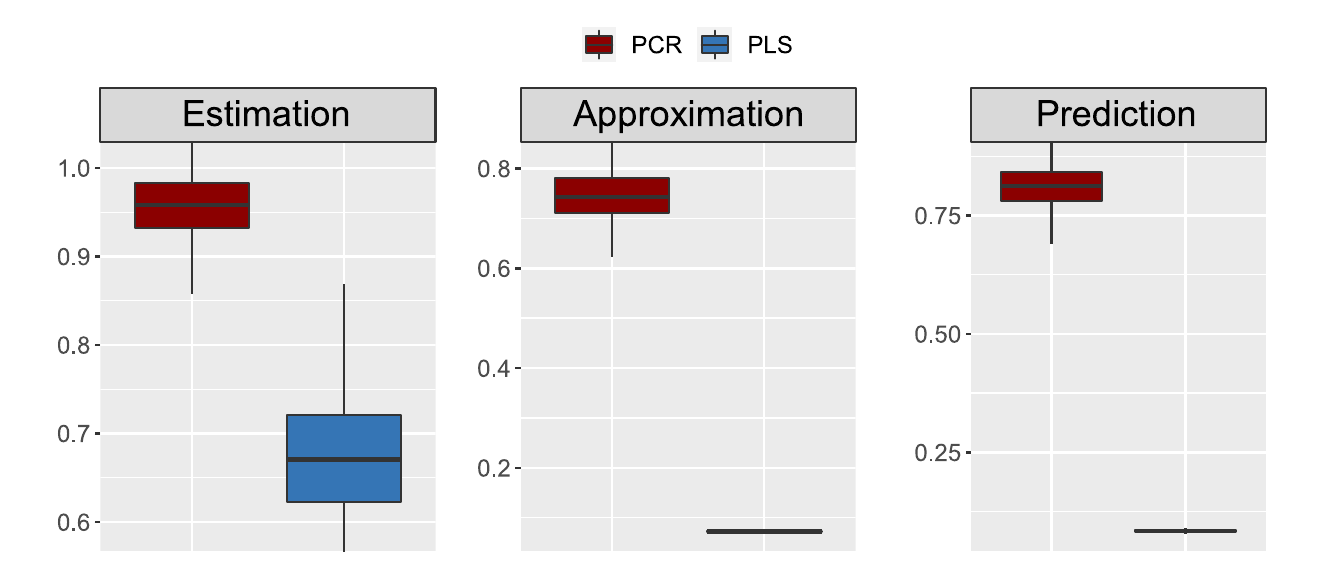}
\vspace{-0.5cm}
\caption{Performance of PCR (red) and PLS (blue). The estimation error $\|\wh{\bbeta}_{\wh{s}}-\bbeta_{0}\|_{2}/\|\bbeta_{0}\|_{2}$, the approximation error $\|\wh{\by}_{\wh{s}}-\by_{train}\|_{2}/\|\by_{train}\|_{2}$, the prediction error $\|\wh{\by}_{\wh{s}}-\by_{test}\|_{2}/\|\by_{test}\|_{2}$. TOP: Weak irrelevant features $\sigma_{y^\bot}=0.1$. BOTTOM: Strong irrelevant features $\sigma_{y^\bot}=2.5$.}
\label{fig:n>p}
\end{figure}

Over $K=500$ repetitions, we split the dataset into two random training and test sets both of size $n/2$. At each repetition, we compute the estimators $\wh{\bbeta}_{\wh{s}}$ for PCR and PLS using $1\leq \wh{s}\leq m$ degrees-of-freedom where $\wh{s}$ is the largest integer $1\leq s\leq m$ for which the condition number $\kappa_{2}(\bX\wh{\bU}_{s})$ is smaller than a chosen threshold $\kappa_{0}>1$ inspired by \cite{kim2019mult}, here we choose $\log_{10}(\kappa_{0})=2.25$. For each method, we compute the relative estimation error $\|\wh{\bbeta}_{\wh{s}}-\bbeta_{0}\|_{2}/\|\bbeta_{0}\|_{2}$, the relative approximation error $\|\wh{\by}_{\wh{s}}-\by_{train}\|_{2}/\|\by_{train}\|_{2}$ on the training data and the relative prediction error $\|\wh{\by}_{\wh{s}}-\by_{test}\|_{2}/\|\by_{test}\|_{2}$ on the test data. We compare in Figure~\ref{fig:n>p} the performance of PCR and PLS in presence of weak/strong irrelevant features with $\sigma_{y^\bot}\in\{0.1,2.5\}$. We can see that PLS is either much better than or comparable with PCR. In particular, the PLS estimator often requires much fewer degrees-of-freedom (not show in the figure).

\subsection{Real Data}\label{sec:num-real}
We revisit the findings of \cite{krivobokova2012} who considered data generated by the MD simulations for the yeast aquaporin (Aqy1), the gated water channel of the yeast \textit{Pichia pastoris}. The data are given as Euclidean coordinates of $N=783$ atoms, thus $p=783\times 3 = 2.349$ features, of Aqy1 observed in a 100 nanosecond time frame, split into $n=20.000$ equidistant observations. Additionally, the diameter of the channel $y_{i}$ is measured by the distance between two centers of mass of certain residues of the protein $\bx_{i}$. We take the first half of the data as training set $(\bX_{train},\by_{train})$ and the remaining half as test set $(\bX_{test},\by_{test})$, each consisting of $n/2=10.000$ observations. Since the data has been produced via molecular dynamics simulations, the observations $(\bx_{i},y_{i})$, $i=1,\ldots,n$, are not independent nor identically distributed and \cite{singer2016partial} show that PLS estimates might be inconsistent if one does not account for this dependence. We thus normalize the training data with an estimated temporal covariance matrix $\wh{\bSigma}\in\R^{n\times n}$ computed according to \cite{Klockmann2024}. That is, we use $(\wt{\bX}_{train},\wt{\by}_{train})$ with $\wt{\bX}_{train}=\wh{\bSigma}^{-1/2}\bX_{train}$ and $\wt{\by}_{train}=\wh{\bSigma}^{-1/2}\by_{train}$. The results are shown in Figure~\ref{fig:aqy1-scale} and discussed below and PLS is confirmed to be the superior method.

\begin{figure}[ht]
    \centering
    \includegraphics[width=\textwidth]{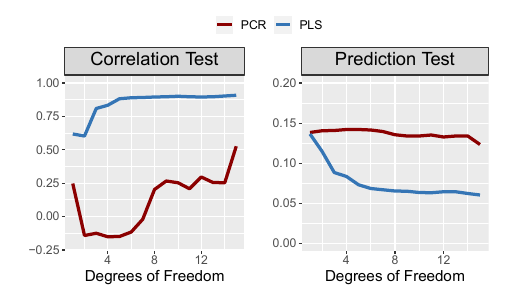}
    \vspace{-1cm}
    \caption{Comparison between PCR and PLS on the Aqy1 dataset studied by \cite{krivobokova2012} rescaled according to \cite{Klockmann2024}. Correlation $\cor(\wh{\by}_{\wh{s}},\by_{test})$ between estimated response and true response on test data and relative $L^2$-prediction error $\|\wh{\by}_{\wh{s}}-\by_{test}\|_{2}/\|\by_{test}\|_{2}$ between estimated response and true response on test data.}
    \label{fig:aqy1-scale}
\end{figure}

From the training set, we compute PCR/PLS estimators $\wh{\bbeta}_{\wh{s}}$ corresponding to $\wh{s}=1,\ldots,15$ latent components. We also compute the estimated condition number $\wh{\kappa}_{\wh{s}}$ of the reduced sample covariance matrix on the training set. To evaluate the models, we compute the correlation between the estimated responses $\wh{\by}_{\wh{s}} = \bX_{test} \wh{\bbeta}_{\wh{s}}$ and the observed response $\by_{test}$ on the test set, together with the relative $L^2$-prediction error $\|\wh{\by}_{\wh{s}}-\by_{test}\|_{2}/\|\by_{test}\|_{2}$. Figure~\ref{fig:aqy1-scale} shows that PCR is much worse than PLS in terms of correlation and prediction on the test data. Even with $\wh{s}=15$, the correlation induced by PCR barely reaches $50\%$, whereas that of PLS is essentially $90\%$. We thus confirm the empirically findings by \cite{krivobokova2012} on their Aqy1 dataset which showed that PCR might be misleading when large directions of variation are uncorrelated with the response.

\section{Discussion} \label{sec:outro}
We provided a novel framework that is compatible with high-dimensional datasets arising from modern applications and we developed the tools to study and compare linear dimensionality reduction algorithms such as PLS and PCR. Our extended latent factor model naturally generalizes to the case where the features are $\bx=\bq+\signot\be+\bx_{y^\bot}$ and the response satisfies $\E(y|\bq)=g(\bq^\top\balpha)$ for some known link function $g$. In a future work, we will address this problem and study the statistical properties of an appropriate generalized-PLS algorithm. A comprehensive theory on the subject is unavailable despite the many heuristic attempts to extend the PLS algorithm to ill-posed generalized linear models due to \cite{marx1996irpls}, \cite{Fort2004}, \cite{ding2005class}, \cite{bastien2005plsglr} and \cite{stocchero2021plsc}. The main challenge is to tackle the additional iterative scheme that is typical of methods computing the sample maximum likelihood such as iteratively-reweighted-least-squares discussed by \cite{mccullagh1989glm}.

\clearpage

\begin{appendix}

\section{Auxiliary Results}\label{app:aux}
Here we gather all the relevant auxiliary results and provide proofs when necessary.

\subsection{Random Vectors}\label{app:aux:rand}

\begin{lemma}[Lemma~B.1 by \cite{finocchio2025mod}]\label{lem:x.y.cov}
    Let $\bx\in\R^p$ be a possibly degenerate random vector and $y\in\R$ a random variable, both centered and with finite second moments. Then, $\bx\in\mR(\bSigma_{\bx})$ almost surely and $\bsigma_{\bx,y}\in\mR(\bSigma_{\bx})$.
\end{lemma}

\begin{lemma}[Lemma~B.2 by \cite{finocchio2025mod}]\label{lem:x.y.ls}
    Let $(\bx,y)\in\R^p\times\R$ be a centered random pair for which the squared-loss $\ell_{\bx,y}(\bbeta):=\E(y-\bx^t\bbeta)^2$ is well-defined for all $\bbeta\in\R^p$. The set of least-squares solutions $\ls(\bx,y,\R^p) := \argmin_{\bbeta\in\R^p} \ell_{\bx,y}(\bbeta)$ is $\{\bbeta\in\R^p:\bSigma_{\bx}\bbeta=\bsigma_{\bx,y}\}$ and the minimum-$L^2$-norm solution is $\bbeta_{\ls} := \bSigma_{\bx}^\dagger\bsigma_{\bx,y}$.
\end{lemma}

\begin{lemma}[Lemma~2.2 by \cite{finocchio2025mod}]\label{lem:ls.rel.pop}
    Let $(\bx,y)\in\R^p\times\R$ satisfy Assumption~\ref{ass:x.y.model.free.2nd}. The relevant subspace $\mB_{y}$ in Equation~\eqref{eq:rel.sub} is unique. Furthermore, with $\bx_{y}$ the relevant features in Equation~\eqref{eq:x.rel.x.irr}, it holds $\ls(\bx,y) = \ls(\bx_{y},y)$ for the population least-squares problem in Equation~\eqref{eq:x.y.ls.pop}. 
\end{lemma}

\subsection{Numerical Perturbation Theory}\label{app:A:pert}
In this section we provide classical and novel results which are relevant to the theory of deterministic perturbations of least-squares problems.

\begin{lemma}[Theorem~3.3.16 by~\cite{Horn1991}]\label{lem:weyl}
    Let $\bA,\bB\in\R_{\succeq0}^{p\times p}$ be any two matrices. Then,
    \begin{align*}
        \lambda_{i+j-1}(\bA+\bB) \leq \lambda_{i}(\bA) + \lambda_{j}(\bB) \leq \lambda_{i+j-p}(\bA+\bB),\quad 1\leq i,j \leq p,
    \end{align*}
    and also
    \begin{align*}
        |\lambda_{i}(\bA+\bB)-\lambda_{i}(\bA)| \leq \lambda_{1}(\bB),\quad 1\leq i\leq p.
    \end{align*}
\end{lemma}

\begin{theorem}[Theorem~1.1 by~\cite{wei1989}]\label{thm:ls.pert}
    Let $\bzeta_{\ls} := \ls(\bA,\bb)$ be the minimum-$L^2$-norm solution of a least-squares problem with $\bA\in\R^{p\times p}$ some symmetric and positive semi-definite matrix and $\bb\in\mR(\bA)$ some vector. Let $\wt{\bzeta}_{\ls} := \ls(\wt{\bA},\wt{\bb})$ be the minimum-$L^2$-norm solution of a perturbed least-squares problem with $\wt{\bA}=\bA+\wt{\Delta\bA}\in\R^{p\times p}$ some symmetric and positive semi-definite matrix and $\wt{\bb}=\bb+\wt{\Delta\bb}\in\mR(\wt{\bA})$ some vector. Assume that $\rank(\wt{\bA})=\rank(\bA)$ and 
    \begin{align*}
        \frac{\|\wt{\Delta\bb}\|_{2}}{\|\bb\|_{2}} \leq \eps,\quad \frac{\|\wt{\Delta\bA}\|_{op}}{\|\bA\|_{op}} \leq \eps,\quad 0 \leq \eps \leq \frac{1}{2\cdot \kappa_{2}(\bA)}.
    \end{align*}
    Then,
    \begin{align*}
        \frac{\|\wt{\bzeta}_{\ls}-\bzeta_{\ls}\|_{2}}{\|\bzeta_{\ls}\|_{2}} &\leq 5 \cdot \kappa_{2}(\bA) \cdot \eps.
    \end{align*}
\end{theorem}

\begin{lemma}\label{lem:pls.resid}
    Let $\bA\in\R^{p\times p}$ be any symmetric positive-semidefinite matrix, $\bb\in\mR(\bA)$ any vector. For all $1\leq s' < s\leq \deg(p_{\bA})$, let $\bzeta_{pls,s}=\pls(\bA,\bb,s)$ and $\bzeta_{pls,s'}=\pls(\bA,\bb,s')$. Then, the residual $\br_{s} = \bzeta_{pls,s} - \bzeta_{pls,s'}$ is orthogonal to the Krylov space $\mK_{s'}(\bA,\bb)$.
\end{lemma}
\begin{proof}[\textbf{Proof of Lemma~\ref{lem:pls.resid}}]
    This is one of the defining properties of the PLS algorithm discussed by \cite{Helland1988pls}. It implies that, with $\bK_{s'}\bK_{s'}^\top$ the orthogonal projection onto the $s'$-dimensional Krylov space $\mK_{s'}(\bA,\bb)$, one has $\bK_{s'}\bK_{s'}^\top \bzeta_{pls,s} = \bzeta_{pls,s'}$ or, equivalently, $\bK_{s'}\bK_{s'}^\top \br_{s} = \bzero_{p}$.
\end{proof}

\section{Proofs}\label{app:proof}
Here we provide all the proofs for the results in the main sections.

\subsection{Proofs for Section~\ref{sec:elf}}\label{app:proof:lm}

\begin{lemma}\label{lem:x.y.elfm}
    Let $(\bx,y)\in\R^p\times\R$ satisfy Assumption~\ref{ass:x.y.model.free.2nd}. Under Equation~\eqref{eq:x.y.elfm} let $\balpha_{\ls}$ be the minimum-$L^2$-norm solution of the population least-squares problem $\ls(\bq,y)$ and $\bU_{\bq}\bbeta_{\ls}$ the orthogonal projection onto $\mR(\bSigma_{\bq})$ of the the population least-squares solution $\bbeta_{\ls}$ from Equation~\eqref{eq:x.y.ls.pop}. Then, $\bU_{\bq}\bbeta_{\ls}$ is also the minimum-$L^2$-norm solution of the population least-squares problem $\ls(\bx_{\bq},y)$. Furthermore, if $\signot^2<\lambda_{r_{\bq}}(\bSigma_{\bq})/2$, then
    \begin{align*}
        \frac{\|\bU_{\bq}\bbeta_{\ls}-\balpha_{\ls}\|_{2}}{\|\balpha_{\ls}\|_{2}} \leq 5\ \frac{\signot^2}{\lambda_{r_{\bq}}(\bSigma_{\bq})}.
    \end{align*}
\end{lemma}
\begin{proof}[\textbf{Proof of Lemma~\ref{lem:x.y.elfm}}]
    To prove the first statement recall the following. From the definition of relevant subspace $\mB_{y}$ in Equation~\eqref{eq:rel.sub} and the corresponding factorization in Equation~\eqref{eq:x.rel.x.irr}, we find that $\bU_{y}=\bSigma_{\bx_{y}}^\dagger\bSigma_{\bx_{y}}$ is the orthogonal projection of $\R^p$ onto $\mR(\bSigma_{\bx_{y}}^\dagger)=\mR(\bSigma_{\bx_{y}})$. Since $\mR(\bU_{\bq})=\mR(\bSigma_{\bq})\subseteq\mB_{y}=\mR(\bSigma_{\bx_{y}})$ is the range of the latent features in Equation~\eqref{eq:x.y.elfm} we also find $\bU_{\bq}\bSigma_{\bx_{y}}^\dagger\bSigma_{\bx_{y}} = \bSigma_{\bx_{y}}^\dagger\bSigma_{\bx_{y}}\bU_{\bq} = \bU_{\bq}$. This implies
    \begin{align*}
        \bSigma_{\bx_{y}}\bU_{\bq} (\bU_{\bq}\bSigma_{\bx_{y}}^\dagger) \bSigma_{\bx_{y}}\bU_{\bq} = \bSigma_{\bx_{y}}\bU_{\bq} \bU_{\bq}\bSigma_{\bx_{y}}^\dagger \bSigma_{\bx_{y}}\bU_{\bq} = \bSigma_{\bx_{y}}\bU_{\bq},
    \end{align*}
    so that $(\bSigma_{\bx_{y}}\bU_{\bq})^\dagger=\bU_{\bq}\bSigma_{\bx_{y}}^\dagger$. Also, with the square-root matrices $\bSigma_{\bx_{y}}^{1/2}$ and $\bSigma_{\bx_{y}}^{\dagger/2}$ of $\bSigma_{\bx_{y}}$ and $\bSigma_{\bx_{y}}^\dagger$ from Theorem~7.2.6 by \cite{Horn1985} one finds as well $\mR(\bSigma_{\bx_{y}}^{1/2})=\mR(\bSigma_{\bx_{y}})$ and $\mR(\bSigma_{\bx_{y}}^{\dagger/2})=\mR(\bSigma_{\bx_{y}}^\dagger)$. With all the above, we can finally check that the projection of $\bbeta_{\ls}$ onto $\mR(\bSigma_{\bq})$ satisfies
    \begin{align*}
        \bU_{\bq}\bSigma_{\bx_{y}}^\dagger\bsigma_{\bx_{y},y} = \bU_{\bq}\bSigma_{\bx_{y}}^{\frac{\dagger}{2}}\bSigma_{\bx_{y}}^{\frac{\dagger}{2}}\bsigma_{\bx_{y},y} = (\bSigma_{\bx_{y}}^{\frac{1}{2}}\bU_{\bq})^\dagger\bSigma_{\bx_{y}}^{\frac{\dagger}{2}}\bsigma_{\bx_{y},y} = (\bU_{\bq}\bSigma_{\bx_{y}}^{\frac{1}{2}}\bSigma_{\bx_{y}}^{\frac{1}{2}}\bU_{\bq})^\dagger \bU_{\bq}\bSigma_{\bx_{y}}^{\frac{1}{2}} \bSigma_{\bx_{y}}^{\frac{\dagger}{2}}\bsigma_{\bx_{y},y}
    \end{align*}
    and the last term in the above display is exactly $(\bU_{\bq}\bSigma_{\bx_{y}}\bU_{\bq})^\dagger \bU_{\bq}\bsigma_{\bx_{y},y} = \bSigma_{\bx_{\bq}}^\dagger\bsigma_{\bx_{\bq},y}$ the minimum-$L^2$-norm solution of $\ls(\bx_{\bq},y)$.

    We now prove the second statement. Under Assumption~\ref{ass:x.y.model.free.2nd} and Equation~\eqref{eq:x.y.elfm} we find $\bSigma_{\bq}^\dagger\bsigma_{\bq,y}$ and $\bSigma_{\bx_{\bq}}^\dagger\bsigma_{\bx_{\bq},y}$ to be solutions of
    \begin{align*}
        \ls(\bq,y) = \ls(\bSigma_{\bq},\bsigma_{\bq,y}),\quad \ls(\bx_{\bq},y) = \ls(\bSigma_{\bq} + \signot^2\bU_{\bq}\bSigma_{\be}\bU_{\bq},\bsigma_{\bx_{\bq},y}).
    \end{align*}
    We now check the assumptions of Theorem~\ref{thm:ls.pert}. First, $\bSigma_{\bq}$ and $\bSigma_{\bx_{\bq}}$ have the same rank. Second, it holds
    \begin{align*}
        \frac{\|\bsigma_{\bx_{\bq},y}-\bsigma_{\bq,y}\|_{2}}{\|\bsigma_{\bq,y}\|_{2}} = 0,\quad  \frac{\|\bSigma_{\bx_{\bq}}-\bSigma_{\bq}\|_{op}}{\|\bSigma_{\bq}\|_{op}} = \signot^2\frac{\lambda_{1}(\bU_{\bq}\bSigma_{\be}\bU_{\bq})}{\lambda_{1}(\bSigma_{\bq})} \leq \frac{\signot^2}{\lambda_{1}(\bSigma_{\bq})} < \frac{1}{2\ \kappa_{2}(\bSigma_{\bq})}.
    \end{align*}
    We can thus apply Theorem~\ref{thm:ls.pert} with $\eps=\signot^2/\lambda_{1}(\bSigma_{\bq})$ and get the claim.
\end{proof}

\begin{lemma}\label{lem:x.y.plr.elfm}
    Let $(\bx,y)\in\R^p\times\R$ satisfy Assumption~\ref{ass:x.y.model.free.2nd}. Under Equation~\eqref{eq:x.y.elfm} let $\balpha_{\ls}$ be the minimum-$L^2$-norm solution of the population least-squares problem $\ls(\bq,y)$ and $\bbeta_{r_{\bq}}$ the $r_{\bq}$-parsimonious parameter for $\mB_{r_{\bq}}$ in Equation~\eqref{eq:x.rel.y.B.s.krylov}. There exists a constant $C_{r_{\bq}}\geq1$ such that, if $\signot^2<\lambda_{r_{\bq}}(\bSigma_{\bq})/\{2 C_{r_{\bq}}+2\}$, then
    \begin{align*}
        \frac{\|\bbeta_{r_{\bq}}-\balpha_{\ls}\|_{2}}{\|\balpha_{\ls}\|_{2}} \leq 5\ \{C_{r_{\bq}}+1\}\ \frac{\signot^2}{\lambda_{r_{\bq}}(\bSigma_{\bq})}.
    \end{align*}
\end{lemma}
\begin{proof}[\textbf{Proof of Lemma~\ref{lem:x.y.plr.elfm}}]
    Without loss of generality, the latent range spans the whole latent Krylov space in the sense that $\mR(\bSigma_{\bq})=\mK_{r_{\bq}}(\bSigma_{\bq},\bsigma_{\bq,y})$. This means that one can rewrite the latent least-squares solution as $\balpha_{\ls}=\balpha_{\pls,r_{\bq}}$ the latent population PLS solution computed by $\pls(\bq,y)$. By Lemma~\ref{lem:x.y.pls.pop} we also can rewrite the $r_{\bq}$-parsimonious parameter as $\bbeta_{r_{\bq}}=\bbeta_{\pls,r_{\bq}}$ the population PLS solution computed from $\pls(\bx_{y},y)$. Under the extended latent factor model in Equation~\eqref{eq:x.y.elfm} we find
    \begin{align*}
        \mK_{r_{\bq}}(\bq,y) = \mK_{r_{\bq}}(\bSigma_{\bq},\bsigma_{\bq,y}),\quad \mK_{r_{\bq}}(\bx_{y},y) = \mK_{r_{\bq}}(\bSigma_{\bq}+\signot^2\bSigma_{\be},\bsigma_{\bq,y}).
    \end{align*}
    We now check that Assumption~2.3 by \cite{finocchio2025mod} holds and we can apply Theorem~2.4 by the same authors. We need to check five conditions. Condition~(i) requires parsimony, this is true because $r_{\bq}=\dim(\mK_{r_{\bq}}(\bq,y)) \leq \dim(\mR(\bSigma_{\bq})) = r_{\bq}$. Condition~(ii) requires stability, this was shown to be true by \cite{finocchio2025mod} so we can always find constants $C_{r_{\bq}}\geq1$, $D_{r_{\bq}}\geq1$ and $M_{r_{\bq}}=2\ \kappa_{2}(\bSigma_{\bq})\ \{C_{r_{\bq}}+1\}$. Condition~(iii) requires compatibility, this is true because $\dim(\mK_{r_{\bq}}(\bq,y))=r_{\bq}=\dim(\mK_{r_{\bq}}(\bx_{y},y))$. Condition~(iv) requires adaptivity, which is true by Lemma~\ref{lem:x.y.pls.pop}. Condition~(v) requires small perturbation error, which is true because
    \begin{align*}
        \frac{\|\bsigma_{\bx_{y},y}-\bsigma_{\bq,y}\|_{2}}{\|\bsigma_{\bq,y}\|_{2}} \vee \frac{\|\bSigma_{\bx_{y}}-\bSigma_{\bq}\|_{op}}{\|\bSigma_{\bq}\|_{op}} &= 0 \vee \frac{\signot^2}{\|\bSigma_{\bq}\|_{op}} < \frac{1}{M_{r_{\bq}}}.
    \end{align*}
    We thus find
    \begin{align*}
        \frac{\|\bbeta_{r_{\bq}}-\balpha_{\ls}\|_{2}}{\|\balpha_{\ls}\|_{2}} \leq 5\ \kappa_{2}(\bSigma_{\bq})\ \{C_{r_{\bq}}+1\}\ \frac{\signot^2}{\|\bSigma_{\bq}\|_{op}},
    \end{align*}
    which is the claim since $\|\bSigma_{\bq}\|_{op}=\lambda_{1}(\bSigma_{\bq})$ and $\kappa_{2}(\bSigma_{\bq})=\lambda_{1}(\bSigma_{\bq})/\lambda_{r_{\bq}}(\bSigma_{\bq})$.
\end{proof}

\begin{lemma}\label{lem:x.rel.y.B.s.latent}
    Let $(\bx,y)\in\R^p\times\R$ satisfy Assumption~\ref{ass:x.y.model.free.2nd}. 
    Under Equation~\eqref{eq:x.y.elfm} assume that $\|\balpha_{\ls}-\bbeta_{r_{\bq}}\|_{\bSigma_{\bq}^2}=\min_{1\leq s\leq m_{y}}\|\balpha_{\ls}-\bbeta_{s}\|_{\bSigma_{\bq}^2}$. Then, the smallest dimension in Equation~\eqref{eq:x.rel.y.B.s.star} satisfies $s_{0}\leq r_{\bq}$.
\end{lemma}
\begin{proof}[\textbf{Proof of Lemma~\ref{lem:x.rel.y.B.s.latent}}]
    With $\mB_{s}=\mK_{s}(\bx_{y},y)$ for all $r_{\bq}\leq s\leq m_{y}$, let $\bU_{s}$ be the orthogonal projection of $\R^p$ onto $\mB_{s}$, we have $\bq_{s}=\bU_{s}\bq$, $\be_{s}=\bU_{s}\be$ and $\bx_{s} = \bU_{s}\bx_{y} = \bq_{s} + \be_{s}$. Since $\bbeta_{s}=\bU_{s}\bbeta_{s}$, we find
    \begin{align*}
        \argmin_{r_{\bq}\leq s\leq m_{y}}\ \E(y-\bx_{s}^t\bbeta_{s})^2 &= \argmin_{r_{\bq}\leq s\leq m_{y}}\ \E(\bq^t\balpha_{\ls}-\bq_{s}^t\bbeta_{s}-\be_{s}^t\bbeta_{s})^2 \displaybreak[0] \\
        &= \argmin_{r_{\bq}\leq s\leq m_{y}}\ \left\{\E(\bq^t\{\balpha_{\ls}-\bbeta_{s}\})^2 + \E(\be^t\bbeta_{s})^2 \right\} \displaybreak[0] \\
        &= \argmin_{r_{\bq}\leq s\leq m_{y}}\ \left\{\|\balpha_{\ls}-\bbeta_{s}\|_{\bSigma_{\bq}}^2 + \|\bbeta_{s}\|_{\bSigma_{\be}}^2 \right\} \displaybreak[0] \\
        &= \argmin_{r_{\bq}\leq s\leq m_{y}}\ \left\{\|\balpha_{\ls}-\bbeta_{s}\|_{\bSigma_{\bq}}^2 + \|\bbeta_{r_{\bq}}\|_{\bSigma_{\be}}^2 + \|\bbeta_{s}-\bbeta_{r_{\bq}}\|_{\bSigma_{\be}}^2 \right\} \displaybreak[0] \\
        &= \argmin_{r_{\bq}\leq s\leq m_{y}}\ \left\{\|\balpha_{\ls}-\bbeta_{s}\|_{\bSigma_{\bq}}^2 + \|\bbeta_{s}-\bbeta_{r_{\bq}}\|_{\bSigma_{\be}}^2 \right\}.
    \end{align*}
    By assumption, the latter display attains minimum at $s=r_{\bq}$, thus the minimizing set is, at its largest, $\{r_{\bq},r_{\bq}+1,\ldots,m_{y}\}$. Therefore, the smallest dimension in Equation~\eqref{eq:x.rel.y.B.s.star} satisfies $s_{0}\leq \min\{r_{\bq},\ldots,m_{y}\} = r_{\bq}$.
\end{proof}

\subsection{Proofs for Section~\ref{sec:pls}}\label{app:proof:pls}

\begin{proof}[\textbf{Proof of Lemma~\ref{lem:x.y.pls.pop}}]
    We recall the definitions 
    \begin{align*}
        \mK_{s}(\bx_{y},y) = \spa\{\bsigma_{\bx_{y},y},\ldots,\bSigma_{\bx_{y}}^{s-1}\bsigma_{\bx_{y},y}\},\quad 
        \mK_{s}(\bx,y) = \spa\{\bsigma_{\bx,y},\ldots,\bSigma_{\bx}^{s-1}\bsigma_{\bx,y}\},
    \end{align*}
    for all $1\leq s\leq p$. From the definition of relevant subspace in Equation~\eqref{eq:rel.sub} and the orthogonal factorization in Equation~\eqref{eq:x.rel.x.irr}, it follows
    \begin{align*}
        \bsigma_{\bx,y} &= \E(\bx y) = \E(\bx_{y} y) \oplus \E(\bx_{y^\bot} y) = \E(\bx_{y} y) \oplus \bzero_{p} = \bsigma_{\bx_{y},y}, \\
        \bSigma_{\bx} &= \E(\bx\bx^t) = \E(\bx_{y}\oplus\bx_{y^\bot})(\bx_{y}\oplus\bx_{y^\bot})^t = \E(\bx_{y}\bx_{y}^t) \oplus \E(\bx_{y^\bot}\bx_{y^\bot}^t) = \bSigma_{\bx_{y}} \oplus \bSigma_{\bx_{y^\bot}}.
    \end{align*}
    Thus, the same holds for $\bSigma_{\bx}^s = (\bSigma_{\bx_{y}} \oplus \bSigma_{\bx_{y^\bot}})^s = \bSigma_{\bx_{y}}^s \oplus \bSigma_{\bx_{y^\bot}}^s$. One last computation yields
    \begin{align*}
        \mK_{s}(\bx,y) &= \spa\{\bsigma_{\bx,y},\ldots,\bSigma_{\bx}^{s-1}\bsigma_{\bx,y}\}, \displaybreak[0] \\
        &= \spa\{\bsigma_{\bx_{y},y},\ldots,\bSigma_{\bx_{y}}^{s-1}\bsigma_{\bx_{y},y} \oplus \bSigma_{\bx_{y^\bot}}^{s-1}\bsigma_{\bx_{y},y} \} \displaybreak[0] \\
        &= \spa\{\bsigma_{\bx_{y},y},\ldots,\bSigma_{\bx_{y}}^{s-1}\bsigma_{\bx_{y},y} \oplus \bzero_{p} \} \displaybreak[0] \\
        &= \mK_{s}(\bx_{y},y),
    \end{align*}
    which is the claim.
\end{proof}

\begin{proof}[\textbf{Proof of Theorem~\ref{thm:x.y.pls.pop}}]
    It follows from Lemma~\ref{lem:x.y.pls.pop} that $\bbeta_{\pls,s_{0}}=\bbeta_{s_{0}}$. This means that
    \begin{align*}
        \frac{\|\bbeta_{\pls,s}-\bbeta_{s_{0}}\|_{2}}{\|\bbeta_{s_{0}}\|_{2}} = \frac{\|\bbeta_{\pls,s}-\bbeta_{\pls,s_{0}}\|_{2}}{\|\bbeta_{\pls,s_{0}}\|_{2}}.
    \end{align*}
    From the orthogonality property of the PLS method, see Lemma~\ref{lem:pls.resid}, for all $1\leq s\leq s_{0}$ the residual $\bbeta_{\pls,s_{0}}-\bbeta_{\pls,s}$ is orthogonal to $\bbeta_{\pls,s}$. This means that we can find an orthonormal basis $\{\bk_{1},\ldots,\bk_{s_{0}}\}$ of $\mK_{s_{0}}(\bx,y)$ such that $\bbeta_{\pls,s_{0}} = \sum_{\ell=1}^{s_{0}} c_{\ell} \bk_{\ell}$ and $\bbeta_{\pls,s} = \sum_{\ell=1}^{s} c_{\ell} \bk_{\ell}$ with the same coefficients. We thus bound,
    \begin{align*}
        \|\bbeta_{\pls,s_{0}}-\bbeta_{\pls,s}\|_{2} = \left(\sum_{\ell=s+1}^{s_{0}} c_{\ell}^2\right)^{\frac{1}{2}} \leq \left(\max_{\ell=s+1,\ldots,s_{0}} c_{\ell}^2\right)^{\frac{1}{2}} \sqrt{s_{0}-s} \leq \|\bbeta_{\pls,s_{0}}\|_{2}\ \sqrt{s_{0}-s}.
    \end{align*}
    We obtain the claim by dividing the above display by $\|\bbeta_{\pls,s_{0}}\|_{2}$.
\end{proof}

\begin{proof}[\textbf{Proof of Theorem~\ref{thm:x.y.pls.pop.elfm}}]
    By Lemma~\ref{lem:x.y.pls.pop} the $r_{\bq}$-dimensional population PLS solution coincides with the $r_{\bq}$-parsimonious parameter $\bbeta_{\pls,r_{\bq}}=\bbeta_{r_{\bq}}\in\mB_{r_{\bq}}$ in Equation~\eqref{eq:x.rel.y.B.s.krylov}. For all $1\leq s\leq r_{\bq}$, we can apply the triangle inequality to get
    \begin{align*}
        \frac{\|\bbeta_{\pls,s}-\balpha_{\ls}\|_{2}}{\|\balpha_{\ls}\|_{2}} &\leq \frac{\|\bbeta_{\pls,s}-\bbeta_{r_{\bq}}\|_{2}}{\|\bbeta_{r_{\bq}}\|_{2}} \cdot \frac{\|\bbeta_{r_{\bq}}\|_{2}}{\|\balpha_{\ls}\|_{2}} + \frac{\|\bbeta_{r_{\bq}}-\balpha_{\ls}\|_{2}}{\|\balpha_{\ls}\|_{2}} \\
        &\leq \frac{\|\bbeta_{\pls,s}-\bbeta_{r_{\bq}}\|_{2}}{\|\bbeta_{r_{\bq}}\|_{2}} \cdot \frac{\|\bbeta_{r_{\bq}}-\balpha_{\ls}\|_{2}+\|\balpha_{\ls}\|_{2}}{\|\balpha_{\ls}\|_{2}} + \frac{\|\bbeta_{r_{\bq}}-\balpha_{\ls}\|_{2}}{\|\balpha_{\ls}\|_{2}} \\
        &= \frac{\|\bbeta_{\pls,s}-\bbeta_{r_{\bq}}\|_{2}}{\|\bbeta_{r_{\bq}}\|_{2}} \cdot \left\{\frac{\|\bbeta_{r_{\bq}}-\balpha_{\ls}\|_{2}}{\|\balpha_{\ls}\|_{2}}+1\right\} + \frac{\|\bbeta_{r_{\bq}}-\balpha_{\ls}\|_{2}}{\|\balpha_{\ls}\|_{2}}.
    \end{align*}
    Since the assumptions of Theorem~\ref{thm:x.y.pls.pop} and Lemma~\ref{lem:x.y.plr.elfm} hold, we can apply them to get
    \begin{align*}
        \frac{\|\bbeta_{\pls,s}-\balpha_{\ls}\|_{2}}{\|\balpha_{\ls}\|_{2}} &\leq \sqrt{r_{\bq}-s} \cdot \left\{5\ \{C_{r_{\bq}}+1\}\ \frac{\signot^2}{\lambda_{r_{\bq}}(\bSigma_{\bq})}+1\right\} + 5\ \{C_{r_{\bq}}+1\}\ \frac{\signot^2}{\lambda_{r_{\bq}}(\bSigma_{\bq})} \\
        &\leq \frac{7}{2}\sqrt{r_{\bq}-s} + 5\ \{C_{r_{\bq}}+1\}\ \frac{\signot^2}{\lambda_{r_{\bq}}(\bSigma_{\bq})},
    \end{align*}
    which is the claim.
\end{proof}

\begin{proof}[\textbf{Proof of Theorem~\ref{thm:x.y.pls.pop.elfm.stop}}]
    Under Equation~\eqref{eq:x.y.elfm} we have moments $\bsigma_{\bx,y}=\bsigma_{\bq,y}$ and $\bSigma_{\bx}=\bSigma_{\bq}+\signot^2\bSigma_{\be}+\bSigma_{\bx_{y^\bot}}$. With $\bU_{r_{\bq}}$ the orthogonal projection of $\R^p$ onto the $r_{\bq}$-parsimonious reduction $\mB_{r_{\bq}}=\mK_{r_{\bq}}(\bSigma_{\bx_{y}},\bsigma_{\bx_{y},y})\subseteq\mB_{y}$ in Equation~\eqref{eq:x.rel.y.B.s.krylov} and the fact that $\bU_{r_{\bq}}\bSigma_{\bx}\bU_{r_{\bq}}=\bU_{r_{\bq}}\bSigma_{\bx_{y}}\bU_{r_{\bq}}$, we can write
    \begin{align*}
        \kappa_{2}(\bU_{r_{\bq}}\bSigma_{\bx}\bU_{r_{\bq}}) &= \kappa_{2}(\bU_{r_{\bq}}\bSigma_{\bx_{\by}}\bU_{r_{\bq}}) = \frac{\lambda_{1}(\bU_{r_{\bq}}\bSigma_{\bq}\bU_{r_{\bq}} + \signot^2\bU_{r_{\bq}}\bSigma_{\be}\bU_{r_{\bq}})}{\lambda_{r_{\bq}}(\bU_{r_{\bq}}\bSigma_{\bq}\bU_{r_{\bq}} + \signot^2\bU_{r_{\bq}}\bSigma_{\be}\bU_{r_{\bq}})}.
    \end{align*}
    We bound the latter display from above by invoking Weyl's inequality in Lemma~\ref{lem:weyl}. We find
    \begin{align}\label{eq:proof.k2.up1}
        \kappa_{2}(\bU_{r_{\bq}}\bSigma_{\bx}\bU_{r_{\bq}}) &\leq \frac{\lambda_{1}(\bU_{r_{\bq}}\bSigma_{\bq}\bU_{r_{\bq}}) + \signot^2\lambda_{1}(\bU_{r_{\bq}}\bSigma_{\be}\bU_{r_{\bq}})}{\lambda_{r_{\bq}}(\bU_{r_{\bq}}\bSigma_{\bq}\bU_{r_{\bq}}) + \signot^2\lambda_{r_{\bq}}(\bU_{r_{\bq}}\bSigma_{\be}\bU_{r_{\bq}})} \leq \frac{\lambda_{1}(\bU_{r_{\bq}}\bSigma_{\bq}\bU_{r_{\bq}}) + \signot^2}{\lambda_{r_{\bq}}(\bU_{r_{\bq}}\bSigma_{\bq}\bU_{r_{\bq}})}.
    \end{align}
    With $\bU_{\bq}$ the orthogonal projection of $\R^p$ onto $\mR(\bSigma_{\bq})$, we find $\mR(\bU_{r_{\bq}})=\mK_{r_{\bq}}(\bSigma_{\bq}+\signot^2\bSigma_{\be},\bsigma_{\bq,y})$ and $\mR(\bU_{\bq})=\mK_{r_{\bq}}(\bSigma_{\bq},\bsigma_{\bq,y})$. Therefore, by definition of constant of stability $C_{r_{\bq}}\geq1$ for population PLS it must be that
    \begin{align*}
        \|\bU_{r_{\bq}}-\bU_{\bq}\|_{op} \leq C_{r_{\bq}}\ \left\{\frac{\|\bsigma_{\bx_{y},y}-\bsigma_{\bq,y}\|_{2}}{\|\bsigma_{\bq,y}\|_{2}} \vee \frac{\|\bSigma_{\bx_{y}}-\bSigma_{\bq}\|_{op}}{\|\bSigma_{\bq}\|_{op}} \right\} = C_{r_{\bq}}\ \frac{\signot^2}{\|\bSigma_{\bq}\|_{op}}.
    \end{align*}
    This implies that    
    \begin{align*}
        \|\bU_{r_{\bq}}\bSigma_{\bq}\bU_{r_{\bq}}-\bU_{\bq}\bSigma_{\bq}\bU_{\bq}\|_{op} \leq 2\ \|\bU_{r_{\bq}}-\bU_{\bq}\|_{op}\ \|\bSigma_{\bq}\|_{op} \leq 2\ C_{r_{\bq}}\ \signot^2.
    \end{align*}
    Invoking again Weyl's inequality in Lemma~\ref{lem:weyl}, together with $\signot^2<\lambda_{r_{\bq}}(\bSigma_{\bq})/\tau\{C_{r_{\bq}}+1\}$ we can further bound
    \begin{align}\label{eq:proof.k2.up2}
        \kappa_{2}(\bU_{r_{\bq}}\bSigma_{\bx}\bU_{r_{\bq}}) &\leq \frac{\lambda_{1}(\bSigma_{\bq}) + 2C_{r_{\bq}}\signot^2 + \signot^2}{\lambda_{r_{\bq}}(\bSigma_{\bq})-2C_{r_{\bq}}\signot^2} < \frac{\lambda_{1}(\bSigma_{\bq}) + \frac{2}{\tau}\lambda_{1}(\bSigma_{\bq})}{\lambda_{r_{\bq}}(\bSigma_{\bq})-\frac{2}{\tau}\lambda_{r_{\bq}}(\bSigma_{\bq})} = \frac{\tau+2}{\tau-2}\ \kappa_{2}(\bSigma_{\bq}).
    \end{align}
    With $\bU_{r_{\bq}+1}$ the orthogonal projection of $\R^p$ onto $\mB_{r_{\bq}+1}=\mK_{r_{\bq}+1}(\bSigma_{\bx_{y}},\bsigma_{\bx_{y},y})\subseteq\mB_{y}$, we now write
    \begin{align*}
        \kappa_{2}(\bU_{r_{\bq}+1}\bSigma_{\bx}\bU_{r_{\bq}+1}) &= \kappa_{2}(\bU_{r_{\bq}+1}\bSigma_{\bx_{y}}\bU_{r_{\bq}+1}) = \frac{\lambda_{1}(\bU_{r_{\bq}+1}\bSigma_{\bq}\bU_{r_{\bq}+1} + \signot^2\bU_{r_{\bq}+1}\bSigma_{\be}\bU_{r_{\bq}+1})}{\lambda_{r_{\bq}+1}(\bU_{r_{\bq}+1}\bSigma_{\bq}\bU_{r_{\bq}+1} + \signot^2\bU_{r_{\bq}+1}\bSigma_{\be}\bU_{r_{\bq}+1})}.
    \end{align*}
    We bound the latter display from below by invoking Weyl's inequality in Lemma~\ref{lem:weyl}. We find
    \begin{align}\label{eq:proof.k2.low1}
        \kappa_{2}(\bU_{r_{\bq}+1}\bSigma_{\bx}\bU_{r_{\bq}+1}) &\geq \frac{\lambda_{1}(\bU_{r_{\bq}+1}\bSigma_{\bq}\bU_{r_{\bq}+1}) + \signot^2\lambda_{r_{\bq}+1}(\bU_{r_{\bq}+1}\bSigma_{\be}\bU_{r_{\bq}+1})}{\lambda_{r_{\bq}+1}(\bU_{r_{\bq}+1}\bSigma_{\bq}\bU_{r_{\bq}+1}) + \signot^2\lambda_{1}(\bU_{r_{\bq}+1}\bSigma_{\be}\bU_{r_{\bq}+1})} \geq \frac{\lambda_{1}(\bU_{r_{\bq}+1}\bSigma_{\bq}\bU_{r_{\bq}+1})}{\signot^2}.
    \end{align}
    Notice in the latter display that the matrix $\bU_{r_{\bq}+1}\bSigma_{\bq}\bU_{r_{\bq}+1}$ has rank $r_{\bq}$. Furthermore, since $\mR(\bU_{r_{\bq}})\subseteq\mR(\bU_{r_{\bq}+1})$ and $\signot^2<\lambda_{r_{\bq}}(\bSigma_{\bq})/\tau\{C_{r_{\bq}}+1\}\leq \lambda_{r_{\bq}}(\bSigma_{\bq})/2\tau$, an application of Weyl's inequality in Lemma~\ref{lem:weyl} gives
    \begin{align}\label{eq:proof.k2.low2}
        \kappa_{2}(\bU_{r_{\bq}+1}\bSigma_{\bx}\bU_{r_{\bq}+1}) &\geq \frac{\lambda_{1}(\bU_{r_{\bq}}\bSigma_{\bq}\bU_{r_{\bq}})}{\signot^2} 
        > \frac{\lambda_{1}(\bSigma_{\bq})-\frac{2}{\tau}\lambda_{1}(\bSigma_{\bq})}{\frac{1}{2\tau}\lambda_{r_{\bq}}(\bSigma_{\bq})} = 2\{\tau-2\}\ \kappa_{2}(\bSigma_{\bq}).
    \end{align}
    Since $\tau\geq8$ implies $\{\tau-2\}/\{\tau+2\}\geq 1/2$, we have shown that
    \begin{align*}
        r_{\bq} \in \mM:=\left\{1\leq s\leq p-1: \frac{\kappa_{2}(\bU_{s+1}\bSigma_{\bx}\bU_{s+1})}{\kappa_{2}(\bU_{s}\bSigma_{\bx}\bU_{s})} > \tau-2 \right\},
    \end{align*}
    thus $m_{\bq}=\min\mM \leq r_{\bq}$. 
\end{proof}

\begin{proof}[\textbf{Proof of Theorem~\ref{thm:x.y.pls.sam.elfm.stop}}]
    For all $1\leq s\leq m_{y}$, pick any sequence $\nu_{\wt{\mB}_{s},n} < \nu_{s,n} < \frac{1}{2}$ and denote $\wh{\Omega}_{s}=\{\wh{\eps}(\bx,y)\leq K_{\wt{\mB}_{s}}\ \nu_{s,n}^{-1}\ \delta_{\wt{\mB}_{s},n}\}$ the event of probability at least $1-2\nu_{s,n}$. From now on, we work on the event $\wh{\Omega}_{r_{\bq}+1}$ which has probability at least $1-2\nu_{r_{\bq}+1,n}$. On this event, we consider $\mR(\wh{\bU}_{r_{\bq}})=\mK_{r_{\bq}}(\wh{\bSigma}_{\bx},\wh{\bsigma}_{\bx,y})$ and $\mR(\bU_{r_{\bq}})=\mK_{r_{\bq}}(\bSigma_{\bx},\bsigma_{\bx,y})$ so that by definition of constant of stability $\wt{C}_{r_{\bq}}\geq1$ for the PLS algorithm we have
    \begin{align*}
        \|\wh{\bU}_{r_{\bq}}-\bU_{r_{\bq}}\|_{op} &\leq \wt{C}_{r_{\bq}}\ \left\{\frac{\|\wh{\bSigma}_{\bx}-\bSigma_{\bx}\|_{op}}{\|\bSigma_{\bx}\|_{op}} \vee \frac{\|\wh{\bsigma}_{\bx,y}-\bsigma_{\bx,y}\|_{2}}{\|\bsigma_{\bx,y}\|_{2}}\right\} \leq \wt{C}_{r_{\bq}}\ K_{\wt{\mB}_{r_{\bq}+1}}\ \nu_{r_{\bq}+1,n}^{-1}\ \delta_{\wt{\mB}_{r_{\bq}}+1,n}.
    \end{align*}
    This implies
    \begin{align*}
        \|\wh{\bU}_{r_{\bq}}\wh{\bSigma}_{\bx}\wh{\bU}_{r_{\bq}}-\bU_{r_{\bq}}\bSigma_{\bx}\bU_{r_{\bq}}\|_{op} &\leq \|\wh{\bSigma}_{\bx}-\bSigma_{\bx}\|_{op} + 2\ \|\wh{\bU}_{r_{\bq}}-\bU_{r_{\bq}}\|_{op}\ \|\bSigma_{\bx}\|_{op} \\
        &\leq 3\ \{\|\bSigma_{\bx}\|_{op}\vee\|\bsigma_{\bx,y}\|_{2}\}\ \wt{C}_{r_{\bq}}\ K_{\wt{\mB}_{r_{\bq}+1}}\ \nu_{r_{\bq}+1,n}^{-1}\ \delta_{\wt{\mB}_{r_{\bq}}+1,n}.
    \end{align*}
    Using $3 \{\|\bSigma_{\bx}\|_{op}\vee\|\bsigma_{\bx,y}\|_{2}\} \wt{C}_{r_{\bq}} K_{\wt{\mB}_{r_{\bq}+1}} \nu_{r_{\bq}+1,n}^{-1} \delta_{\wt{\mB}_{r_{\bq}}+1,n} < \lambda_{r_{\bq}}(\bSigma_{\bq})/2\tau$, Weyl's inequality in Lemma~\ref{lem:weyl} and Equations~\eqref{eq:proof.k2.up1}~-~\eqref{eq:proof.k2.up2}, we can bound from above
    \begin{align*}
        \kappa_{2}(\wh{\bU}_{r_{\bq}}\wh{\bSigma}_{\bx}\wh{\bU}_{r_{\bq}}) &= \frac{\lambda_{1}(\wh{\bU}_{r_{\bq}}\wh{\bSigma}_{\bx}\wh{\bU}_{r_{\bq}})}{\lambda_{r_{\bq}}(\wh{\bU}_{r_{\bq}}\wh{\bSigma}_{\bx}\wh{\bU}_{r_{\bq}})} \displaybreak[0] \\
        &< \frac{\lambda_{1}(\bU_{r_{\bq}}\bSigma_{\bx}\bU_{r_{\bq}}) + \frac{1}{2\tau}\lambda_{1}(\bSigma_{\bq})}{\lambda_{r_{\bq}}(\bU_{r_{\bq}}\bSigma_{\bx}\bU_{r_{\bq}}) - \frac{1}{2\tau}\lambda_{r_{\bq}}(\bSigma_{\bq})} \displaybreak[0] \\
        &\leq \frac{\lambda_{1}(\bU_{r_{\bq}}\bSigma_{\bq}\bU_{r_{\bq}}) + \signot^2 + \frac{1}{2\tau}\lambda_{1}(\bSigma_{\bq})}{\lambda_{r_{\bq}}(\bU_{r_{\bq}}\bSigma_{\bq}\bU_{r_{\bq}}) - \frac{1}{2\tau}\lambda_{r_{\bq}}(\bSigma_{\bq})} \displaybreak[0] \\
        &\leq \frac{\lambda_{1}(\bSigma_{\bq}) + 2C_{r_{\bq}}\signot^2 + \signot^2 + \frac{1}{2\tau}\lambda_{1}(\bSigma_{\bq})}{\lambda_{r_{\bq}}(\bSigma_{\bq}) - 2C_{r_{\bq}}\signot^2 - \frac{1}{2\tau}\lambda_{r_{\bq}}(\bSigma_{\bq})} \displaybreak[0] \\
        &\leq \frac{\lambda_{1}(\bSigma_{\bq}) + \frac{2}{\tau}\lambda_{1}(\bSigma_{\bq}) + \frac{1}{2\tau}\lambda_{1}(\bSigma_{\bq})}{\lambda_{r_{\bq}}(\bSigma_{\bq}) - \frac{2}{\tau}\lambda_{r_{\bq}}(\bSigma_{\bq}) - \frac{1}{2\tau}\lambda_{r_{\bq}}(\bSigma_{\bq})} \displaybreak[0] \\
        &= \frac{2\tau+5}{2\tau-5}\ \kappa_{2}(\bSigma_{\bq}).
    \end{align*}
    On the same event, we can repeat the argument for $\mR(\wh{\bU}_{r_{\bq}+1})=\mK_{r_{\bq}+1}(\wh{\bSigma}_{\bx},\wh{\bsigma}_{\bx,y})$ and $\mR(\bU_{r_{\bq}+1})=\mK_{r_{\bq}+1}(\bSigma_{\bx},\bsigma_{\bx,y})$. Using $3 \{\|\bSigma_{\bx}\|_{op}\vee\|\bsigma_{\bx,y}\|_{2}\} \wt{C}_{r_{\bq}+1} K_{\wt{\mB}_{r_{\bq}+1}} \nu_{r_{\bq}+1,n}^{-1} \delta_{\wt{\mB}_{r_{\bq}+1},n} < \lambda_{r_{\bq}}(\bSigma_{\bq})/2\tau$, Weyl's inequality in Lemma~\ref{lem:weyl} and Equations~\eqref{eq:proof.k2.low1}~-~\eqref{eq:proof.k2.low2}, we can bound from below
    \begin{align*}
        \kappa_{2}(\wh{\bU}_{r_{\bq}+1}\wh{\bSigma}_{\bx}\wh{\bU}_{r_{\bq}+1}) &= \frac{\lambda_{1}(\wh{\bU}_{r_{\bq}+1}\wh{\bSigma}_{\bx}\wh{\bU}_{r_{\bq}+1})}{\lambda_{r_{\bq}+1}(\wh{\bU}_{r_{\bq}+1}\wh{\bSigma}_{\bx}\wh{\bU}_{r_{\bq}+1})} \displaybreak[0] \\
        &> \frac{\lambda_{1}(\bU_{r_{\bq}+1}\bSigma_{\bx}\bU_{r_{\bq}+1}) - \frac{1}{2\tau}\lambda_{1}(\bSigma_{\bq})}{\lambda_{r_{\bq}+1}(\bU_{r_{\bq}+1}\bSigma_{\bx}\bU_{r_{\bq}+1}) + \frac{1}{2\tau}\lambda_{r_{\bq}}(\bSigma_{\bq})} \displaybreak[0] \\
        &\geq \frac{\lambda_{1}(\bSigma_{\bq}) - \frac{2}{\tau}\lambda_{1}(\bSigma_{\bq}) - \frac{1}{2\tau}\lambda_{1}(\bSigma_{\bq})}{\frac{1}{2\tau}\lambda_{r_{\bq}}(\bSigma_{\bq}) + \frac{1}{2\tau}\lambda_{r_{\bq}}(\bSigma_{\bq})} \displaybreak[0] \\
        &= \frac{2\tau-5}{2}\ \kappa_{2}(\bSigma_{\bq}).
    \end{align*}
    Since $\tau\geq8$ implies $\{2\tau-5\}/\{2\tau+5\}\geq1/2$, we have shown that with probability at least $1-2\nu_{r_{\bq}+1,n}$,
    \begin{align*}
        r_{\bq} \in \wh{\mM}:=\left\{1\leq s\leq p-1: \frac{\kappa_{2}(\wh{\bU}_{s+1}\wh{\bSigma}_{\bx}\wh{\bU}_{s+1})}{\kappa_{2}(\wh{\bU}_{s}\wh{\bSigma}_{\bx}\wh{\bU}_{s})} > \frac{2\tau-5}{4} \right\}
    \end{align*}
    so that $\wh{m}_{\bq}=\min\wh{\mM} \leq r_{\bq}$.
\end{proof}

\end{appendix}

\clearpage

\bibliographystyle{plainnat}       
\bibliography{bibPRO}           

\end{document}